\documentclass[11pt]{amsart}

\usepackage{tikz-cd} 

\usepackage[english]{babel}
\newcommand{\suchthat}{\;\ifnum\currentgrouptype=16 \middle\fi|\;}
\usepackage[utf8x]{inputenc}
\usepackage[T1]{fontenc}
\usepackage{amsmath}
\usepackage{amstext}
\usepackage{amsfonts}
\usepackage{amssymb}
\usepackage{amsthm}
\usepackage{amsrefs}
\usepackage{enumitem}
\usepackage[all]{xy}
\usepackage{mathtools}
\usepackage{tikz-cd}
\usepackage{dsfont}
\usepackage{color}
\usepackage{graphicx}
\usepackage[colorinlistoftodos]{todonotes}
\usepackage[colorlinks=true, linkcolor=red,
urlcolor=blue]{hyperref}
\usepackage{float}
\usepackage[colorinlistoftodos]{todonotes}
\usepackage{theoremref}
\newcounter{TmpEnumi}
\allowdisplaybreaks
\theoremstyle{definition}
\newtheorem{thm}{Theorem}[section]
\newtheorem{lem}[thm]{Lemma}
\newtheorem{prp}[thm]{Proposition}
\newtheorem{dfn}[thm]{Definition}

\newtheorem{ctn}[thm]{Construction}
\newtheorem{rmk}[thm]{Remark}
\newtheorem{ntn}[thm]{Notation}

\newtheorem{qst}[thm]{Question}

\newcommand{\rc}{{\operatorname{rc}}}

\newcommand{\Ad}{\operatorname{Ad}}

\newcommand{\bN}{{\mathbb{N}}}
\newcommand{\bZ}{{\mathbb{Z}}}

\newcommand{\id}{\operatorname{id}}

\newcommand{\Aut}{{\operatorname{Aut}}}
\newcommand{\dirlim}{\varinjlim}
\numberwithin{equation}{section}
\newcommand\numeq[1]%
  {\stackrel{\scriptscriptstyle(\mkern-1.5mu#1\mkern-1.5mu)}{=}}

\title[The radius of comparison]{The radius of comparison for actions of 
$\mathbb{Z}^d$ on simple AH algebras}

\author{M. Ali Asadi-Vasfi}
\address{Department of Mathematics, University of Toronto and The Fields 
Institute.}
\author{Ilan Hirshberg}
\address{Department of Mathematics, Ben-Gurion University of the Negev.}
\author{Apurva Seth}
\address{Department of Mathematics, Ben-Gurion University of the Negev and 
Mathematical Institute, University of Oxford.}

\thanks{The second and third authors were partially supported by BSF. The third 
author was also partially supported by the Engineering and Physical Sciences 
Research Council (EP/X026647/1). For the purpose of open access, the authors 
have 
applied a CC BY public copyright license to any author-accepted manuscript 
arising from this submission.}

\begin{document}
\maketitle

\begin{abstract}
Given  $0\leq r' \leq r\leq \infty$, and $d\in \bN$, we construct a simple unital 
AH algebra $A$ with stable rank one, and a pointwise outer action $\alpha 
\colon 
\bZ^d\to\text{Aut}(A)$ such that $\rc (A)=r$ and $\rc 
\left(A\rtimes_{\alpha}\bZ^d\right)=r'$. 
\end{abstract}
\section{Introduction}
The radius of comparison is a numerical invariant for unital stably finite 
C*-algebras 
introduced by Toms in \cite{Tom06} as a means of measuring the extent to which 
strict comparison fails. This invariant was used to construct a 
family of pairwise non-isomorphic simple AH algebras with the same Elliott 
invariant. Although considerable work has focused on identifying conditions 
under which the regularity properties required for Elliott classification are 
preserved under natural constructions, little is known about how the radius of 
comparison behaves under these constructions. With the Elliott 
program now largely complete, it is natural to go beyond classifiable cases 
and gain a better understanding of the radius of comparison as an invariant 
rather than merely as a tool for distinguishing examples.

In \cite{hirshberg2024exotic}, the second author constructed an action of 
the circle group \(\mathbb{T}\) on a simple separable unital stably finite 
\(\mathcal{Z}\)-stable nuclear C*-algebra \(A\), such that the crossed 
product \(A \rtimes \mathbb{T}\) is again simple but not 
\(\mathcal{Z}\)-stable. This construction arises as the dual of an action 
of the 
integers on a simple non-\(\mathcal{Z}\)-stable AH algebra with a positive 
radius of comparison;  in that case, the crossed product by \(\mathbb{Z}\) 
becomes 
\(\mathcal{Z}\)-stable. This result raises questions about how the 
radius of comparison behaves under discrete group actions. In the setting of  
simple unital C*-algebras, the results in the literature so far provide bounds 
for the radius of comparison of crossed products by finite group actions, 
assuming the actions in question have either the weak tracial Rokhlin property 
(\cite{AGP19}), or a weak form of approximate representability (\cite{ASV23}).

In this paper, we construct examples demonstrating that the radius of 
comparison can 
decrease to any intermediate value between \(0\) and the original value. More 
precisely, we prove the following theorem.
\begin{thm}\label{thm:main}
	For any pair of extended real numbers \(0 \leq r' \leq r \leq \infty\) and 
	for any $d \in \bN$ there exists a simple AH algebra \(A\) with satble rank 
	one and a pointwise outer action \(\alpha : \mathbb{Z}^d 
	\to \text{Aut}(A)\), such that \(\text{rc}(A) = r\) and \(\text{rc}(A 
	\rtimes_{\alpha} \mathbb{Z}^d) = r'\). 
\end{thm}
\begin{rmk}\label{rmk:r'=0}
	The case $r'=0$ is essentially contained in \cite{hirshberg2024exotic}, and 
	therefore we explicitly prove only the case $r'>0$ in this paper. More 
	specifically, that construction in \cite{hirshberg2024exotic} produced an 
	AH algebra $A$ with positive radius of comparison $r$ and a pointwise outer 
	action of the integers such that the crossed product is
	$\mathcal{Z}$-stable, and in particular has radius of comparison $0$. It 
	provided examples with $r \in (0,1)$. Although $r$ was not 
	computed precisely (only a lower bound was needed), a careful examination 
	using the same arguments as in this paper shows that any value of $r$ can 
	be 
	obtained. The modifications required to achieve $r \geq 1$, and to replace 
	$\bZ$ with $\bZ^d$, are identical to those in our construction and are 
	straightforward. In principle, one could write a unified proof addressing 
	the case $r'=0$, but we chose not to do so to keep the notation manageable 
	in an already technical proof.
\end{rmk}
This shows that there is no simple 
formula relating the radius of comparison of a C*-algebra with that of its 
crossed product. We do not know whether the radius of comparison can increase 
when forming crossed products by pointwise outer actions.

\begin{qst}
	Is there a simple unital stably finite C*-algebra $A$ along with a 
	pointwise outer 
	action $\alpha$ of a countable discrete amenable group $G$ such that $\rc ( 
	A \rtimes_{\alpha} G ) > \rc (A)$?
\end{qst}

We now outline the proof strategy. The proof combines ideas from the 
construction of the example from 
\cite{hirshberg2024exotic} with the diagram merging technique used in 
\cite{HP19,HirPhi23}.  For simplicity, we focus on the case of actions of 
$\bZ$. Given $r>0$, we start out with a construction 
of a simple AH algebra with radius of comparison $r$, following the basic 
setup from \cite{villadsen}. Here the spaces $X_n$ are products of copies of 
$S^2$, the solid arrows indicate coordinate projection maps, and the dotted 
arrows indicate a small number of point evaluations, needed to make the direct 
limit 
simple.

\begin{tikzcd}
	C(X_0) 
	\arrow[r, Rightarrow]\arrow[r,dotted, shift right=2] & C(X_1)\otimes M_{r(1)} \arrow[r]
	\arrow[r, Rightarrow]\arrow[r,dotted, shift right=2] & C(X_2)\otimes M_{r(2)}
	\arrow[r, Rightarrow]\arrow[r,dotted, shift right=2]&\hdots
\end{tikzcd}

To introduce an action, we overlay this diagram on a binary tree.  We introduce point evaluations going across branches of the tree to ensure simplicity. Our action 
 is just a modification of the odometer action: at each level, it is a cyclic 
 permutation of the summands, composed with an inner automorphism that is 
 needed to ensure that the actions at different levels are consistent.

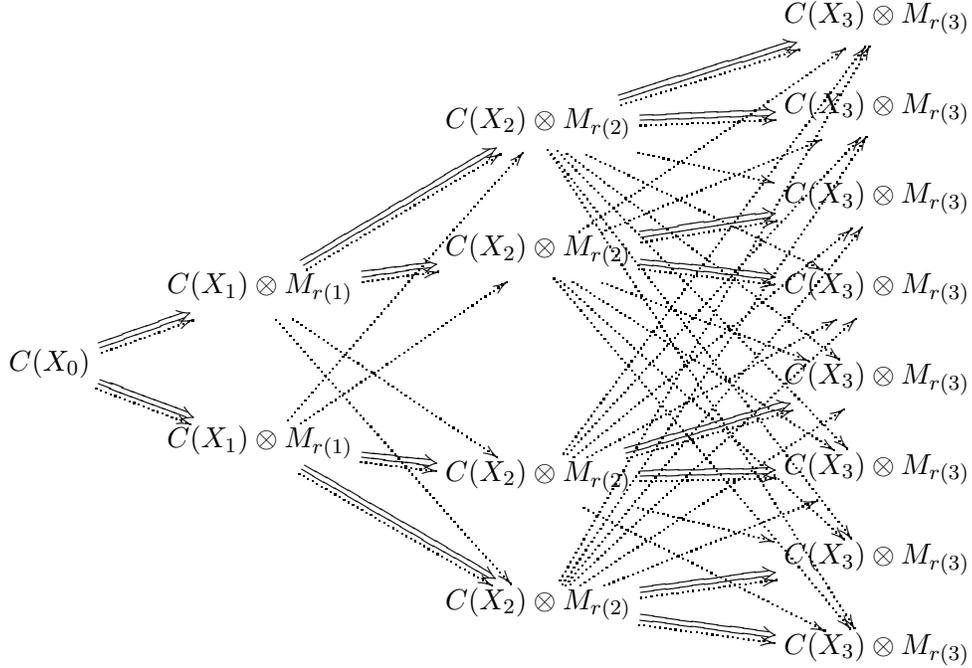
\begin{equation}
	\begin{xy}
		(0,10)*+{C(X_0) }="base";
		(28,-0.5)*+{ C(X_1) \otimes M_{r(1)} }="0";
		(28,20)*+{ C(X_1) \otimes M_{r(1)} }="1";
		(65,-22)*+{ C(X_2) \otimes M_{r(2)} }="00";
		(65,-5)*+{ C(X_2) \otimes M_{r(2)} }="01";
		(65,42)*+{ C(X_2) \otimes M_{r(2)} }="10";
		(65,25)*+{ C(X_2) \otimes M_{r(2)} }="11";
		(110,-28)*+{ C(X_3) \otimes M_{r(3)} }="000";
		(110,-16)*+{ C(X_3) \otimes M_{r(3)} }="001";
		(110,-4)*+{ C(X_3) \otimes M_{r(3)} }="010";
		(110,8)*+{ C(X_3) \otimes M_{r(3)} }="011";
		(110,56)*+{ C(X_3) \otimes M_{r(3)} }="100";
		(110,44)*+{ C(X_3) \otimes M_{r(3)} }="101";
		(110,32)*+{ C(X_3) \otimes M_{r(3)} }="110";
		(110,20)*+{ C(X_3) \otimes M_{r(3)} }="111";
		{\ar@2 "base";"0"};
		{\ar@2 "base";"1"};
		{\ar@2 "0";"01"};
		{\ar@2 "0";"00"};
		{\ar@2 "1";"11"};
		{\ar@2 "1";"10"};
		{\ar@2 "00";"000"};
		{\ar@2 "00";"001"};
		{\ar@2 "01";"010"};
		{\ar@2 "01";"011"};
		{\ar@2 "10";"100"};
		{\ar@2 "10";"101"};
		{\ar@2 "11";"110"};
		{\ar@2 "11";"111"};
		{\ar@{.>}@<-3pt> "base";"0"};
		{\ar@{.>}@<-3pt> "base";"1"};
		{\ar@{.>}@<-3pt> "0";"01"};
		{\ar@{.>}@<-3pt> "0";"00"};
		{\ar@{.>}@<-3pt> "0";"10"};
		{\ar@{.>}@<-3pt> "0";"11"};
		{\ar@{.>}@<-3pt> "1";"01"};
		{\ar@{.>}@<-3pt> "1";"00"};
		{\ar@{.>}@<-3pt> "1";"10"};
		{\ar@{.>}@<-3pt> "1";"11"};
		{\ar@{.>}@<-3pt> "00";"000"};
		{\ar@{.>}@<-3pt> "00";"001"};
		{\ar@{.>}@<-3pt> "00";"010"};
		{\ar@{.>}@<-3pt> "00";"011"};
		{\ar@{.>}@<-3pt> "00";"100"};
		{\ar@{.>}@<-3pt> "00";"101"};
		{\ar@{.>}@<-3pt> "00";"110"};
		{\ar@{.>}@<-3pt> "00";"111"};
		{\ar@{.>}@<-3pt> "01";"000"};
		{\ar@{.>}@<-3pt> "01";"001"};
		{\ar@{.>}@<-3pt> "01";"010"};
		{\ar@{.>}@<-3pt> "01";"011"};
		{\ar@{.>}@<-3pt> "01";"100"};
		{\ar@{.>}@<-3pt> "01";"101"};
		{\ar@{.>}@<-3pt> "01";"110"};
		{\ar@{.>}@<-3pt> "01";"111"};
		{\ar@{.>}@<-3pt> "10";"000"};
		{\ar@{.>}@<-3pt> "10";"001"};
		{\ar@{.>}@<-3pt> "10";"010"};
		{\ar@{.>}@<-3pt> "10";"011"};
		{\ar@{.>}@<-3pt> "10";"100"};
		{\ar@{.>}@<-3pt> "10";"101"};
		{\ar@{.>}@<-3pt> "10";"110"};
		{\ar@{.>}@<-3pt> "10";"111"};
		{\ar@{.>}@<-3pt> "11";"000"};
		{\ar@{.>}@<-3pt> "11";"001"};
		{\ar@{.>}@<-3pt> "11";"010"};
		{\ar@{.>}@<-3pt> "11";"011"};
		{\ar@{.>}@<-3pt> "11";"100"};
		{\ar@{.>}@<-3pt> "11";"101"};
		{\ar@{.>}@<-3pt> "11";"110"};
		{\ar@{.>}@<-3pt> "11";"111"};
	\end{xy}
\end{equation}
The resulting action has the Rokhlin property and is almost periodic. It 
therefore follows from \cite[Theorem~4.1]{hirshberg_JOT} that the crossed 
product has to absorb the Jiang-Su algebra, and in particular, its radius of 
comparison is $0$.

To get intermediate values of radius of comparison, we further complicate the picture. Given $0<r' \leq r$, we construct another AH algebra with radius of comparison $r'$, in a similar manner.

\begin{tikzcd}
	C(Y_0)
	\arrow[r, Rightarrow]\arrow[r,dotted, shift right=2] & C(Y_1)\otimes M_{r(1)} 
	\arrow[r, Rightarrow]\arrow[r,dotted, shift right=2] & C(Y_2)\otimes M_{r(2)} 
	\arrow[r, Rightarrow]\arrow[r,dotted, shift right=2] &\hdots
\end{tikzcd}

We now merge both diagrams, with a small number of point evaluation maps going across:

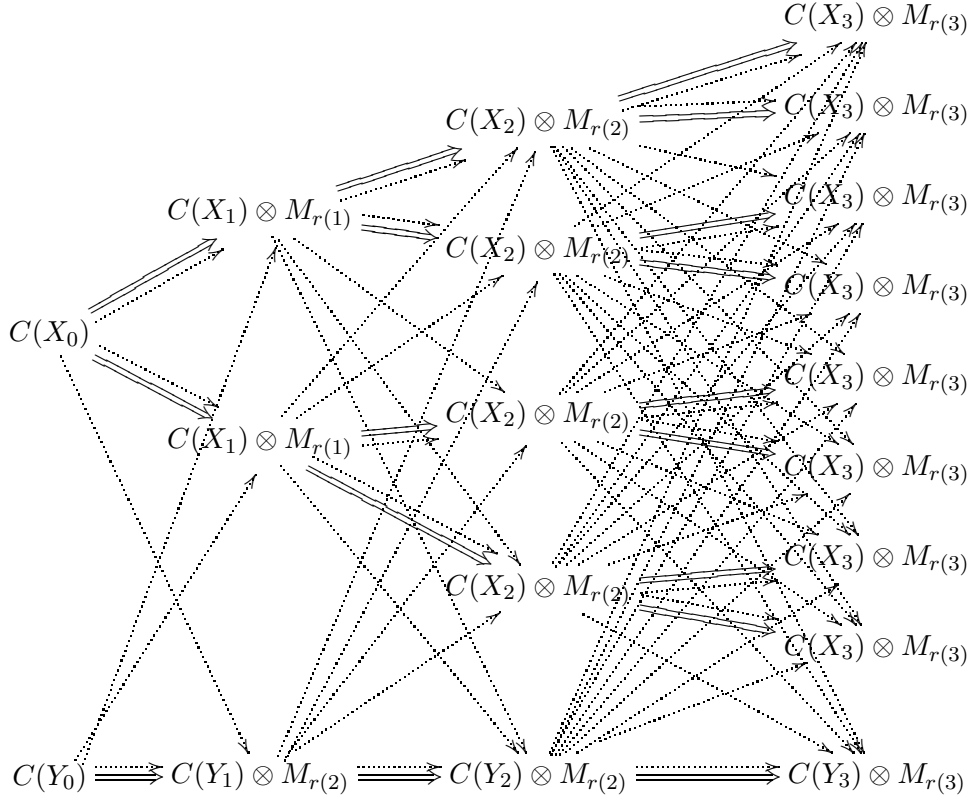
\begin{equation}\label{OurDiag}
	\begin{xy}
		(0,14)*+{C(X_0) }="base";
		(0,-45)*+ { C(Y_0)}= "Base";
		(28,-0.5)*+{ C(X_1) \otimes M_{r(1)} }="0";
		(28,30)*+{ C(X_1) \otimes M_{r(1)} }="1";
		(28,-45)*+ {C(Y_1)\otimes M_{r(2)}}="2";
		(65,-20)*+{ C(X_2) \otimes M_{r(2)} }="00";
		(65,3)*+{ C(X_2) \otimes M_{r(2)} }="01";
		(65,42)*+{ C(X_2) \otimes M_{r(2)} }="10";
		(65,25)*+{ C(X_2) \otimes M_{r(2)} }="11";
		(65,-45)*+{C(Y_2)\otimes M_{r(2)}}="02";
		(110,-28)*+{ C(X_3) \otimes M_{r(3)} }="000";
		(110,-16)*+{ C(X_3) \otimes M_{r(3)} }="001";
		(110,-4)*+{ C(X_3) \otimes M_{r(3)} }="010";
		(110,8)*+{ C(X_3) \otimes M_{r(3)} }="011";
		(110,56)*+{ C(X_3) \otimes M_{r(3)} }="100";
		(110,44)*+{ C(X_3) \otimes M_{r(3)} }="101";
		(110,32)*+{ C(X_3) \otimes M_{r(3)} }="110";
		(110,20)*+{ C(X_3) \otimes M_{r(3)} }="111";
		(110,-45)*+{C(Y_3)\otimes M_{r(3)}}="002";
		{\ar@2 "base";"0"};
		{\ar@2 "Base"; "2"};
		{\ar@2 "2"; "02"};
		{\ar@2 "02"; "002"};
		{\ar@{.>}@< 4pt> "02"; "002"};
		{\ar@{.>}@< 4pt> "Base"; "2"};
		{\ar@{.>} @<0pt> "base"; "2"};
		{\ar@{.>} @<-10pt> "Base"; "1"};
		{\ar@{.>}@<-3pt>"2"; "01"};
		{\ar@{.>}@<1pt>"2"; "00"};
		{\ar@{.>} @<-3pt>"2"; "10"};
		{\ar@{.>} @<-5pt>"2"; "11"};
		{\ar@{.>}"0"; "02"};
		{\ar@{.>}"1"; "02"};
		{\ar@{.>} @<-5pt> "Base"; "0"};
		{\ar@{.>} @< 4pt>"2"; "02"};
		{\ar@{.>}@<1ex> "base";"0"};
		{\ar@2 "base";"1"};
		{\ar@{.>}@<-1ex> "base";"1"};
		{\ar@2 "0";"01"};
		{\ar@{.>}@<-1ex> "0";"01"};
		{\ar@2 "0";"00"};
		{\ar@{.>}@<1ex> "0";"00"};
		{\ar@2 "1";"11"};
		{\ar@{.>}@<1ex> "1";"11"};
		{\ar@2 "1";"10"};
		{\ar@{.>}@<-1ex> "1";"10"};
		{\ar@{.>}
			"0";"10"};
		{\ar@{.>} "0";"11"};
		{\ar@{.>}
			"1";"01"};
		{\ar@{.>}
			"1";"00"};
		{\ar@2 "00";"000"};
		{\ar@{.>}@<1ex> "00";"000"};
		{\ar@{.>}@<1ex> "02";"000"};
		{\ar@2 "00";"001"};
		{\ar@{.>}@<-1ex> "00";"001"};
		{\ar@{.>} "02";"001"};
		{\ar@2 "01";"010"};
		{\ar@{.>}@<1ex> "01";"010"};
		{\ar@{.>} "02";"010"};
		{\ar@2 "01";"011"};
		{\ar@{.>} "02";"011"};
		{\ar@{.>}@<-1ex> "01";"011"};
		{\ar@2 "10";"100"};
		{\ar@{.>} "02";"100"};
		{\ar@{.>}@<-1ex> "10";"100"};
		{\ar@2 "10";"101"};
		{\ar@{.>} "02";"101"};
		{\ar@{.>}@<1ex> "10";"101"};
		{\ar@2 "11";"110"};
		{\ar@{.>} "02";"110"};
		{\ar@{.>}@<-1ex> "11";"110"};
		{\ar@2 "11";"111"};
		{\ar@{.>} "00";"002"};
		{\ar@{.>} "01";"002"};
		{\ar@{.>} "10";"002"};
		{\ar@{.>} "11";"002"};
		{\ar@{.>}@<1ex> "11";"111"};
		{\ar@{.>} "02";"111"};
		{\ar@{.>} "00";"010"};
		{\ar@{.>}
			"00";"011"};	
		{\ar@{.>}
			"00";"111"};
		{\ar@{.>}
			"00";"110"};
		{\ar@{.>}
			"00";"101"};
		{\ar@{.>}
			"00";"100"};
		{\ar@{.>}	"01";"000"};
		{\ar@{.>}	"01";"001"};
		{\ar@{.>}	"01";"100"};
		{\ar@{.>}	"01";"101"};
		{\ar@{.>}	"01";"110"};
		{\ar@{.>}	"01";"111"};
		{\ar@{.>}	"10";"000"};
		{\ar@{.>}	"10";"001"};
		{\ar@{.>}	"10";"010"};
		{\ar@{.>}	"10";"011"};
		{\ar@{.>}	"10";"110"};
		{\ar@{.>}	"10";"111"};
		{\ar@{.>}	"11";"000"};
		{\ar@{.>}	"11";"001"};
		{\ar@{.>}	"11";"100"};
		{\ar@{.>}	"11";"101"};
		{\ar@{.>}	"11";"010"};
		{\ar@{.>}	"11";"011"};
	\end{xy}
\end{equation}
The radius of comparison of the resulting AH algebra is shown to be the maximum 
of the two. Up to composing with inner automorphisms, which are needed to 
ensure consistency, the actions on the top levels are cyclic permutations, 
whereas the action on the bottom part is trivial. In essence, we are merging an 
action which is a souped-up version of the odometer action - which has the 
Rokhlin property - with an inner action on the bottom part.Since no general 
result applies here, we analyze the crossed product directly: the 
crossed product is an AH algebra too, and its radius of comparison is shown to 
be the same as the one coming from the bottom row. 

The rest of the paper constitutes the proof of Theorem~\ref{thm:main}. In light 
of Remark~\ref{rmk:r'=0}, we only prove explicitly the case in which $0<r' \leq 
r \leq \infty$. We divide the proof into sections. In 
Section~\ref{sec:preliminaries}, we recall some definitions and prove an 
elementary lemma about the existence of sequences of natural numbers with a 
prescribed growth condition. In Section~\ref{sec:construction_algebra} we 
construct our AH algebra $A$, along with the action of $\bZ^d$. In 
Section~\ref{sec:rc(A)} we compute the radius of comparison of $A$, and in 
Section~\ref{sec:rc-crossed-product} we compute the radius of comparison of the 
crossed product.
\begin{rmk}
	It follows from \cite[Theorem~1.1]{hirshberg_JOT} that an action on a 
	non-$\mathcal{Z}$-stable C*-algebra cannot simultaneously be approximately 
	inner and have the Rokhlin property. The action we construct has neither 
	property. However, it has some form of an interpolated property. Let us 
	focus on the case of $d=1$ for this remark. At the $n$'th stage of the 
	construction, we obtain a direct sum of $2^n$ copies of $C(X_n) \otimes 
	M_{r(n)}$ and a single copy of $C(Y_n) \otimes M_{r(n)}$. Denote by 
	$p_{n,1},p_{n,2},\ldots,p_{n,2^n}$ the units of the $2^n$ copies of $C(X_n) 
	\otimes M_{r(n)}$. Set $p_n = \sum_{k=1}^{2^n}p_{n,k}$, and set $q_n = 
	1-p_n$. These projections are in the center of the $n$'th level 
	of the direct system, so they form a central sequence. We would have had 
	the tracial Rokhlin property if the sequence $q_n$ tended to zero in 
	the tracial sense, whereas it would be tracially approximately inner if the 
	sequence $p_n$ tended to zero in the tracial sense. In our 
	construction, neither condition holds. Instead, we obtain Rokhlin-like 
	towers whose 
	sum is a central projection that is neither large nor small. It would be 
	interesting to investigate this phenomenon in more generality, but we do 
	not pursue this direction further in the paper.
\end{rmk}

\section{Preliminaries}\label{sec:preliminaries}

If $A$ is a C*-algebra,  we denote by 
$A_{+}$ the
set of positive elements of $A$. Using the natural embedding $M_n(A)\hookrightarrow M_{n+1}(A)$, we take $M_\infty(A)=\bigcup_{n=1}^\infty M_n(A)$.

We recall the definition of Cuntz subequivalence (going back to \cite{JC}) and 
the radius of comparison. 

\begin{dfn}
Let $A$ be a C*-algebra. For $a,b\in M_\infty(A)_{+}$, we say that $a$ is 
\emph{Cuntz subequivalent} to $b$, denoted by $a\lesssim b$, if there is a 
sequence $\left(v_n\right)_{n=1}^\infty$ in $M_\infty(A)$  such that
\[
\lim_{n\to\infty}v_nbv_n^*=a.
\]
\end{dfn}
Let $A$ be unital and stably finite. Let $T(A)$ be the tracial state space of 
$A$. Any $\tau\in T(A)$ extends to a non-normalized trace on $M_n(A)$, for 
$n\in\bN$. We use $\tau$ for the extended trace as 
well. For $a\in M_{n}(A)_{+}$, set
$d_{\tau}(a)=\lim_{n\to\infty}\tau(a^{1/n})$.

The following is the definition of the radius of comparison for unital exact 
stably finite C*-algebras. We refer the reader to \cite{Tom06} for further 
details. 

\begin{dfn}

Let $A$ be a unital stably finite C*-algebra. 

\begin{enumerate}
\item
 For $r\in[0,\infty)$, $A$ is said to have \emph{$r$-comparison} if for any 
 $n\in\{0,1,2,\ldots\}$ and $a,b\in M_n(A)_{+}$ satisfying 
 $d_\tau(a)+r<d_\tau(b)$ 
 for all $\tau\in T(A)$, we have $a\lesssim b$.
\item 
  The radius of comparison, denoted by $\rc(A)$, is the infimum of 
  $r\in[0,\infty)$ such that $A$ has $r$-comparison. If there is no such $r$, 
  one sets $\rc(A)=\infty$.
  \end{enumerate}
\end{dfn} 
  In general, the definition of the radius of comparison requires quasitraces 
  rather than traces. However, the C*-algebras that we consider in this paper 
  are nuclear, and in particular exact, so all quasitraces are traces by 
  \cite[Theorem 5.11]{Hag14}.

We recall the following simple facts.
\begin{prp}\cite[Proposition 6.2]{Tom06}\label{prop_roc}
Let $A$ and $B$ be stably finite unital C*-algebras. The following hold:
\begin{enumerate}
\item $\rc(A\oplus B)=\max\left\{\text{rc}(A),\text{rc}(B)\right\}.$
\item $\rc(M_n\otimes A)=\frac{1}{n}\rc(A),\,\, \text{for}\,\,n\in\bN.$
\end{enumerate}
\end{prp}

We'll need the following lemma about infinite products. 
(See similar lemmas, used for similar purposes, in \cite[Lemma 5.7]{ASV23} and 
\cite[Lemmas 3.9, 3.10]{HirPhi23}.)

\begin{lem}\label{seq_lem_zd}
Let $0 < r' \leq  r<1$, and $d\in\bN$. There exist sequences 
$(d(n))_{n\in\bN}$ and $(d'(n))_{n\in\bN}$ in $\bN$ such that 
\begin{enumerate}[label=(\alph*)]
\item $d(n)$ is nondecreasing and $\lim_{n\rightarrow\infty}d(n)=\infty$.\label{first_sq_lem}
\item $\displaystyle{\prod_{n=1}^{\infty}\frac{d(n)}{d(n)+1+2^{dn-d}}=r}$.\label{sec_sq_lem}
\item $d'(n)\leq d(n)$ for all $n\in\bN$.\label{third_sq_lem}
\item $\displaystyle{\prod_{n=1}^{\infty}\frac{d'(n)}{d(n)+1+2^{dn-d}}}=r'$.\label{fourth_sq_lem}
\end{enumerate}
\end{lem}

\begin{proof}

For $0<r<1$, we first construct a sequence $(d(n))_{n\in\bN}$ in $\bN$ 
satisfying \ref{first_sq_lem} and \ref{sec_sq_lem}. 

We prove this by induction on $n$. Define
 \[d(1)=\min\left(\left\{k\in\bN\suchthat \frac{k}{k+2}>r\right\}\right).\]
 
The set in the above definition is non-empty. Fix
 \[
 r_1=\frac{d(1)}{d(1)+2} .
 \]
 Then $1> r_1>r$. For $n\geq 2$, set 
 \[
 d(n)=\min\left(\left\{k\in\bN\suchthat 
 \frac{k}{k+1+2^{dn-d}}>\frac{r}{r_{n-1}}\right\}\right)\quad\text{and}\quad 
 r_n=\prod_{k=1}^{n}\frac{d(k)}{d(k)+1+2^{dk-d}}.
 \]
 
 We now show that
 \begin{enumerate}
 \item $1 > r_{n}\geq r_{n+1}>r$ for all $n\in\bN$,\label{first_eq_sq_lem}
 \item $d(n)\leq d(n+1)$ for all $n\in\bN$,\label{second_eq_sq_lem}
 \item $\lim_{n\rightarrow\infty}r_n=r$.\label{third_eq_sq_lem}
 \end{enumerate}
 
 (\ref{first_eq_sq_lem}) is immediate. For (\ref{second_eq_sq_lem}), note that
 \[
\frac{r}{r_{n-1}}\leq\frac{r}{r_n}<\frac{d(n+1)}{d(n+1)+1+2^{dn}}<\frac{d(n+1)}{d(n+1)+1+2^{dn-d}}. 
 \]
 
 Thus, \[d(n+1)\in\left\{k\in\bN\suchthat 
 \frac{k}{k+1+2^{dn-d}}>\frac{r}{r_{n-1}}\right\}.\]
 
Since $d(n)$ is defined to be the minimum value of the above set, we conclude that $d(n)\leq d(n+1)$.
We now prove (\ref{third_eq_sq_lem}). For this, by (\ref{first_eq_sq_lem}), we 
see that $(r_n)_{n\in\bN}$ is a decreasing sequence and bounded below by $r$. 
Thus, $(r_n)_{n\in\bZ_{n>0}}$ has a limit, which we denote $s$. We know that $s 
\geq r$. We need to show that $r \geq s$. Since 
\[
{\displaystyle \prod_{n=1}^{\infty}\frac{d(n)}{d(n)+1+2^{dn-d}}}=s,
\]
we get 
\[
\sum_{n=1}^{\infty}\frac{1+2^{dn-d}}{d(n)+1+2^{dn-d}}<\infty.
\]
Thus, 
$\displaystyle{\lim_{n\rightarrow\infty}\frac{2^{dn-d}+1}{d(n)+1+2^{dn-d}}=0}$, 
which implies, $\displaystyle{\lim_{n\rightarrow 
\infty}\frac{d(n)}{1+2^{dn-d}}=\infty}$.
Furthermore, 
\[
\lim_{n\rightarrow\infty}\frac{2^{dn-d}+1}{d(n)+2^{dn-d}}=0.
\]
Assume, to obtain a contradiction, that $r<s$. Then, 
$\displaystyle{0<\frac{r}{s}<1}$. Choose $m\in\bN$ 
large enough, such that 
\[
\frac{2^{dm-d}+1}{d(m)+2^{dm-d}}<1-\frac{r}{s}.
\]
Then, $d(m)\geq 2$ and \[
\frac{r}{s}<1-\frac{2^{dm-d}+1}{d(m)+2^{dm-d}}=\frac{d(m)-1}{(d(m)-1)+1+2^{dm-d}}
. 
\]
 Since $s\leq r_{m-1}$, 

\[
\frac{r}{r_{m-1}}\leq \frac{r}{s}<\frac{d(m)-1}{(d(m)-1)+1+2^{dm-d}}.
\]
This contradicts the definition of $d(m)$.

We now construct the sequence $(d'(n))_{n\in\bN}$ in $\bN$. If $r'=r$, then we 
can simply choose $d'(n) = d(n)$, so we assume from now on that $r' < r$. 

For $n\in\bN$, set 
\[
\rho_n=\prod_{k={n+1}}^{\infty}\frac{d(k)}{d(k)+1+2^{dk-d}} .
\]
Then $\rho_n < 1$ for all $n\in\{0,1,2,\ldots\}$.

We construct a sequence $(t(n))_{n\in\bN}$ by induction such that 
$t(n)(1+2^{dn-d}) \in \bN$ for all $n$, such that
\[
1\leq t(n)(1+2^{dn-d})\leq d(n),
\]
and such that, with 
\[
\gamma_n=\prod_{k=1}^{n}\frac{t(k)(1+2^{dk-d})}{d(k)+1+2^{dk-d}} ,
\] we have
\[
\gamma_n\rho_n\geq r'\quad \text{and}\quad \rho_n\gamma_n-r'<\frac{1}{d(n)+1+2^{dn-d}}.
\]

Note that, since $\lim_{n\rightarrow\infty}(d(n)+1+2^{dn-d})=\infty$ and  $\lim_{n\rightarrow\infty}\rho_n=1$, we conclude that $\lim_{n\rightarrow\infty}\gamma_n=r'$. Defining $d'(n)=t(n)(1+2^{dn-d})$ gives the conclusion of the lemma.

For the base case, let $m$ be the least integer such that
\[
\frac{m\rho_1}{d(1)+2}\geq r'.
\]

Then $m\geq 1$. Fix $\displaystyle{t(1)=\frac{m}{2}}$. Notice that,
\[
\frac{d(1)\rho_1}{d(1)+2}
=\frac{d(1)}{d(1)+2}\prod_{k=2}^{\infty}\frac{d(k)}{d(k)+1+2^{dk-d}}
=\prod_{k=1}^{\infty}\frac{d(k)}{d(k)+1+2^{dk-d}}=\rho_0=r>r'.
\]

Thus, we have $1\leq 2t(1)\leq d(1)$. Moreover
\[
\frac{(m-1)\rho_1}{d(1)+2}<r'.
\]

Since $\rho_1\leq 1$, we have
\[
0\leq \frac{m\rho_1}{d(1)+2}-r'<\frac{\rho_1}{d(1)+2}\leq \frac{1}{d(1)+2}.
\]

Note that 
\[\gamma_1=\frac{2t(1)}{d(1)+2}=\frac{m}{d(1)+2} ,
\] 
so
\[
 \gamma_1\rho_1-r'<\frac{1}{d(1)+2}. 
\]

Now suppose $t(1),t(2),\hdots,t(n)$ were selected, and that, with $\gamma_n$ as 
above, we have $\gamma_n\rho_n\geq r'$ and 
\[
\gamma_n\rho_n-r'<\frac{1}{d(n)+1+2^{dn-d}} .
\]

Let $m$ be the least integer such that 
\[
\frac{m\gamma_{n}\rho_{n+1}}{d(n+1)+1+2^{dn}}\geq r'.
\]

Since $r'>0$, $\gamma_n>0$ and $\rho_{n+1}>0$, we have $m\geq 1$. Fix 
\[
t(n+1)=\frac{m}{1+2^{dn}} .
\] 

See that
\begin{equation*}
\begin{split}
\frac{d(n+1)\gamma_n\rho_{n+1}}{d(n+1)+1+2^{dn}}&=\frac{\gamma_n d(n+1)}{d(n+1)+1+2^{dn}}\prod_{k=n+2}^{\infty}\frac{d(k)}{d(k)+1+2^{dk-d}}\\
&= \gamma_n\prod_{k=n+1}^{\infty}\frac{d(k)}{d(k)+1+2^{dk-d}}\\
&= \gamma_n\rho_n\geq r'.
\end{split}
\end{equation*}

Thus, $1\leq m\leq d(n+1)$, so $ 1\leq t(n+1)(1+2^{dn})\leq d(n+1)$. Moreover,
\[
\frac{(m-1)\gamma_n\rho_{n+1}}{d(n+1)+1+2^{dn}}<r'.
\]

Since $\rho_n,\gamma_n\leq 1$, we have

\[
\frac{m\gamma_n\rho_{n+1}}{d(n+1)+1+2^{dn}}-r'<\frac{\gamma_n\rho_{n+1}}{d(n+1)+1+2^{dn}}<\frac{1}{d(n+1)+1+2^{dn}}.
\]

Using the fact that $m=t(n+1)(1+2^{dn})$ and $\displaystyle{\gamma_{n+1}=\frac{t(n+1)(1+2^{dn})\gamma_n}{d(n+1)+1+2^{dn}}}$, 
we get
\[
\gamma_{n+1}\rho_{n+1}-r'<\frac{1}{d(n+1)+1+2^{dn}}.
\]
This concludes the induction.
\end{proof}

\section{Construction}\label{sec:construction_algebra}
Let $0<r'\leq r\leq \infty$ and $d\in\bN$.
Our goal is to construct a simple AH algebra $A$ and an action 
$\alpha:\bZ^d\rightarrow\text{Aut}(A)$ such that $\rc (A)=r$ and $\rc (A 
\rtimes_{\alpha} \bZ^d )= r'$.  
The construction was outlined in the introduction. We now proceed to make it 
 precise. 

\begin{ctn}\label{OurConstruction_zd}
We define the following objects. 
\begin{enumerate}
\item \label{OurConstruction.1}
Let $0< r'\leq r\leq\infty$, and $d\in\bN$ be given. Let 
$(h_n)_{n\in\bN\cup\{0\}}$ and $(h'_n)_{n\in\bN\cup\{0\}}$ be sequences in 
$\bN$ such that
\begin{itemize}
	\item If $0< r'\leq r<\infty$, then we let $h$ be the least natural number 
	such that $r<h$, and set $h_n=h'_n=h$ for all $n\in\bN$.
	\item If $0< r' < r=\infty$, then 
	\begin{enumerate}
		\item $h'_n \equiv h$, where $h$ is the least natural number such that 
		$r'<h$, 
		and
		\item $h_n$ is chosen such that  $h_0=1$, 
		$\lim_{n\to\infty}h_n=\infty$, and 
		\[\lim_{n\to\infty}\frac{h_n}{2^{nd}}=0.\]
		
	\end{enumerate}
\item If $r'=r=\infty$, we set $h'_n=h_n$, and choose this sequence so that 
$h_0=1$, 
$\lim_{n\to\infty}h_n=\infty$, and \[\lim_{n\to\infty}\frac{h_n}{2^{nd}}=0.\]
\end{itemize}

\item 
Choose sequences of natural numbers $\left(d(n)\right)_{n\in\bN}$ and 
$\left(d'(n)\right)_{n\in\bN}$ such that

\begin{itemize}
	\item If $0< r'\leq r<\infty$, and $h$ is the  least natural number such 
	that 
	$r<h$, $\kappa'=\frac{r'}{h}$, and $\kappa=\frac{r}{h}$, then  $d(n)$ and 
	$d'(n)$ are chosen to satisfy conditions 
	\ref{first_sq_lem}, \ref{sec_sq_lem}, \ref{third_sq_lem} and 
	\ref{fourth_sq_lem} of Lemma \ref{seq_lem_zd} for $0< \kappa'\leq \kappa<1$.
	\item If $0< r'<r=\infty$, and $h$ is the  least natural number such that 
	$r'<h$, and $\kappa'=\frac{r'}{h}$, then $d(n)$ are chosen to satisfy 
	conditions  \ref{first_sq_lem} and \ref{sec_sq_lem} of Lemma 
	\ref{seq_lem_zd} for $0< \kappa'<1$, and we set $d'(n)=d(n)$.
	\item If $r'=r=\infty$, we fix some constant $0< c<1$ and take $d(n)=d'(n)$ 
	satisfying conditions  \ref{first_sq_lem} and \ref{sec_sq_lem} of Lemma 
	\ref{seq_lem_zd}.
\end{itemize} 

\item Set $l(0)= r(0)=s(0)=s'(0)=1$. For $n\in\bN$, set 
\begin{itemize}
\item $\displaystyle{l(n)=d(n)+1+2^{dn-d}}$,
 \item  $r(n)=\prod_{j=1}^n l(j)$,
 \item $s(n)=\prod_{j=1}^{n}d(j)$,
 \item $s'(n)=\prod_{j=1}^{n}d'(j)$.
\end{itemize}

\item \label{OurConstruction.2}
For $n\in\{0,1,2,\ldots\}$, we set 
\[
X_n=\left(S^2\right)^{h_n s(n)} \text{ and } 
Y_n=\left(S^2\right)^{h'_ns'(n)}.
\]
 We identify 
 \[X_{n+1}=\left(S^2\right)^{h_{n+1}s(n+1)}=X_n^{d(n+1)}\times  \left(S^2\right)^{(h_{n+1}-h_n)s(n+1)},\]
  and similarly
  \[Y_{n+1}=Y_n^{d'(n+1)}\times  \left(S^2\right)^{(h'_{n+1}-h'_n)s'(n+1)}.\]
  
  For $n\in\{0,1,2,\ldots\}$, 
$j\in\{1,2,\hdots,d(n+1)\}$, 
and $i\in\{1,2,\hdots,d'(n+1)\}$, let 
\[P_j^{n}:X_{n+1}\rightarrow X_n \text{ and } 
Q_i^n:Y_{n+1}\rightarrow Y_n
\]
 be the $j$-th and the $i$-th coordinate 
projections, respectively.

\item \label{Constuction_choice_of_points} Choose points $x_m\in X_m$ and $y_m\in Y_m$ for $m\in\{0,1,2,\ldots\}$,   
such 
that for all $n\in\{0,1,2,\ldots\}$, the set
\begin{align}\label{con_simple}
\begin{split}
&\Big\{\left(P_{\nu_1}^{(n)}\circ P_{\nu_2}^{(n+1)}\circ\hdots\circ P_{\nu_{m-n}}^{(m-1)}\right)(x_m) \suchthat m>n,\,\, \text{and}\,\,\\ 
 & \nu_j\in\{1,2,\hdots,d(n+j)\}\,\, \text{for}\,\, j=1,2,\hdots,m-n\Big\}
 \end{split}
 \end{align}
 is dense in $X_n$, and the set
\begin{align*}
\begin{split}
&\Big\{\left(Q_{\nu_1}^{(n)}\circ Q_{\nu_2}^{(n+1)}\circ\hdots\circ Q_{\nu_{m-n}}^{(m-1)}\right)(y_m) \suchthat  m>n,\,\, \text{and}\,\,\\ 
 & \nu_i\in\{1,2,\hdots,d'(n+i)\}\,\, \text{for}\,\, i=1,2,\hdots,m-n\Big\}
 \end{split}
 \end{align*} 
 is dense in $Y_n$.
 
 \item \label{OurConstruction.3}
 For $n\in\{0,1,2,\ldots\}$, define 
 \[
 C_n\cong C(X_n\times \bZ_{2^n}^d,M_{r(n)}),\quad \text{and}\quad B_n\cong 
 C(Y_n,M_{r(n)}).
 \]
 We freely identify $C_n$ with \[C_n = M_{r(n)}\otimes C(X_n\times 
 \bZ_{2^n}^d)\] and $B_n$ with \[B_n = M_{r(n)}\otimes C(Y_n).\]
 We also identify 
 \[
 M_{r(n)}\cong M_{l(1)}\otimes M_{l(2)}\otimes\hdots\otimes M_{l(n)}.
 \]
 For $n\in\{0,1,2,\ldots\}$, denote 
 \begin{equation}\label{equ_An}
 A_n=C_n\oplus B_n.
 \end{equation}
 
 \item \label{OurConstruction.4}
 Let $\pi_{n,{n+1}}:\bZ_{2^{n+1}}^d\to \bZ_{2^n}^d$ denote the map 
 \[\pi_{n,{n+1}}(k_1,k_2,\hdots,k_d)=(k_1(\text{mod}\,2^n), 
 k_2(\text{mod}\,2^n),\hdots, k_d(\text{mod}\,2^n) ).\]
 
 \item \label{OurConstruction.5}
 For $n\in\{0,1,2,\ldots\}$, we fix 
 continuous functions
 \[
 S_{n,1},S_{n,2},\hdots,S_{n,{d(n+1)}}:X_{n+1}\times \bZ_{2^{n+1}}^d\rightarrow 
 X_n\times \bZ_{2^n}^d,
 \]
 as follows: 
 for $k\in \bZ_{2^{n+1}}^d$ and for $x\in X_{n+1}$, set  
 \[
 S_{n,j}(x,k)=\left(P_{j}^{(n)}(x),\pi_{n,n+1}(k)\right).
 \] 
 
 \item We define the following index sets \[
 L(n)=\{\bZ_{2^{n-1}}^d\}\bigsqcup\{\bigstar\}\bigsqcup\{1,2,\hdots,d(n)\} \, .
 \]
 Note that the cardinality of $L(n)$ is $l(n)$; indexing the set with those 
 symbols rather than with numbers would make some notation easier.
  Let $\{E_{z,z'}^{n}\}_{z,z'\in L(n)}\in B(\ell^2(L(n)))$ be the standard 
  matrix units, that is, the operators that act on the basis 
  $\{\delta_z^{n}, z\in L(n)\}$ of $\ell^2(L(n))$ by
  \[
E_{z,z'}^{n}(\delta_{z''}^{n})=\begin{cases}
 \delta_z^n, & \text{for } z'=z'',\\
        0, & \text{otherwise}.
 \end{cases}
 \]
To lighten notation, we freely identify $B(\ell^2(L(n)))$ with $M_{l(n)}$ by 
thinking of the 
coordinates as being indexed by the set $L(n)$.

 \item \label{OurConstruction.8_zd}
  We define unital homomorphisms
  
 \begin{eqnarray*}
  \gamma_n:C(X_n\times\bZ_{2^n}^d)\,\oplus \, C(Y_n)\rightarrow \left(C(X_{n+1}\times\bZ_{2^{n+1}}^d)\oplus C(Y_{n+1})\right)\otimes B(\ell^2(L(n+1)))
 \end{eqnarray*}
 by
 \begin{align}\label{def_gamma_n}
 \begin{split}
 \gamma_n(f,g)= &{} \Bigg(\sum_{z\in \bZ_{2^n}^d} f(x_n,z)\otimes E^{n+1}_{z,z} + 
 g(y_n)\otimes E^{n+1}_{\bigstar,\bigstar}+ \sum_{z=1}^{d(n+1)} f\circ 
 S_{n,z}\otimes E^{n+1}_{z,z},\\ 
 &{} \quad \sum_{z\in \bZ_{2^n}^d} f(x_n,z)\otimes E^{n+1}_{z,z}+ g(y_n)\otimes 
 E^{n+1}_{\bigstar,\bigstar}+ \sum_{z=d'(n+1)+1}^{d(n+1)} g(y_n)\otimes 
 E^{n+1}_{z,z} \\ &{} \quad + \sum_{z=1}^{d'(n+1)}g\circ Q_z^n\otimes 
 E^{n+1}_{z,z}\Bigg).
 \end{split}
 \end{align}

  We then define 
 \[
 \Gamma_{n+1,n}:A_n\rightarrow A_{n+1}
 \]
 by $\Gamma_{n+1,n}=\text{id}_{M_{r(n)}}\otimes \gamma_n$. For $n\geq m\geq 0$, 
 set
 \[
 \Gamma_{n,m}=\Gamma_{n,n-1}\Gamma_{n-1,n-2}\circ\hdots\circ\Gamma_{m+1,m}:A_m\rightarrow A_n.
 \]
 
 \item 
 Set $A=\displaystyle{\lim_{\longrightarrow}}(A_n , \Gamma_{n+1,n})$. The maps 
 associated with the direct limit are denoted 
 $\Gamma_{\infty,m}:A_m\rightarrow A$ for $m\geq 0$.\label{ah_algebra}
 
 \item \label{action_def_first_step}
 For $g\in \bZ^d$, define unitaries $w_g^{n}\in B(\ell^2(L(n)))$ as follows. 
 For $z \in \bZ_{2^{n-1}}^d$ and $g= \bZ^d$, we write 
 \[
 z+g=z+ (g +  2^{n-1}\bZ^d) \in  \bZ_{2^{n-1}}^d .
 \] 
 We then set
 \[
    w_g^{n}(\delta_z^{n}) = \begin{cases}
        \delta^n_{z+g}, & \text{for } z\in \bZ_{2^{n-1}}^d,\\
        \delta^n_z, & \text{for } z\in L(n)\setminus \bZ_{2^{n-1}}^d.
        \end{cases}
  \]
  
 We then have 
 \begin{equation}\label{relation_ug_Ez}
\text{Ad}(w_g^n)\circ E^n_{z,z}=\begin{cases}
 E^n_{z+g,z+g}, & \text{for } z\in \bZ_{2^{n-1}}^d,\\
        E^n_{z,z}, & \text{for } z\in L(n)\setminus \bZ_{2^{n-1}}^d.
 \end{cases}
 \end{equation}
 
 Define $u_g^n\in M_{r(n)}$ to be 
 \[
u_g^n = w_g^1\otimes w_g^2\otimes \hdots \otimes w_g^{n}.
 \]

 Now, for $g\in \bZ^d$ define 
 \[
 \sigma^{n}_g:C_n\to C_n,\quad\text{and}\quad \beta^{n}_g:B_n\to B_n
 \]
 by 
 \begin{itemize}
 \item $\sigma^{n}_g(f)(x,k)=u_g^n f(x,k-g) (u_g^n)^*$.
 \item $\beta^n_g(h)(y)=u_g^n h(y) (u_g^n)^*$.
 \end{itemize}

Set $\alpha^{n}_g=\sigma^{n}_g\oplus \beta^{n}_g :A_n\to A_n$. This defines an 
action $\alpha^{n}$ of $\bZ^d$ on $A_n$.
\end{enumerate}

\end{ctn}

\begin{prp}\label{L_af_proper_zd}
Assume the notation and choices in Construction~\ref{OurConstruction_zd}.
\begin{enumerate}
\item\label{L_af_proper.a_zd}
 The C*-algebra $A$ is separable, simple, and has stable rank one.
\item\label{L_af_proper.b_zd}
There is a unique action 
$\alpha \colon \mathbb{Z}^d \to \Aut (A)$ 
such that $\alpha = \dirlim \alpha^{n}$. 
\item\label{L_af_proper.c_zd}
$\alpha$ is pointwise outer.
\item\label{L_af_proper.d_zd}
$A\rtimes_{\alpha} \bZ^d$ is simple. 
\setcounter{TmpEnumi}{\value{enumi}}
\end{enumerate}
\end{prp}

\begin{proof}
We prove (\ref{L_af_proper.a_zd}). 
Separability of $A$ is immediate.
Simplicity follows from Proposition 2.1(ii) of \cite{DNN92}, because of the choice of points in Construction~\ref{OurConstruction_zd}(\ref{Constuction_choice_of_points}). 
Since the direct system in
Construction~\ref{OurConstruction_zd}(\ref{OurConstruction.8_zd})
has diagonal maps in the sense of Definition~2.1 of~\cite{ElHoTm}, it follows from Theorem~4.1 of~\cite{ElHoTm}
that $A$ has stable rank one.

To prove (\ref{L_af_proper.b_zd}), it is enough to show that for $g\in\bZ^d$, 
we have
\begin{equation}\label{action_well_defined}
\alpha^{n+1}_g\circ \Gamma_{n+1,n}=\Gamma_{n+1,n}\circ\alpha^{n}_g.
\end{equation}
Fix a point $(x,k)\in X_{n+1}\times\bZ_{2^{n+1}}^d$ and a point $y\in Y_{n+1}$. 
It 
suffices to show that if
$a,b\in M_{r(n)}$, $f\in C(X_n\times 
\bZ_{2^n}^d)$ and $h\in C(Y_n)$, then
 (\ref{action_well_defined}) holds for $a\otimes 
f\in C_n$ and $b\otimes h\in B_n$. First, observe that for $z\in \{1,2,\hdots, 
d(n+1)\}$, we have
\begin{equation}\label{action_snz_qnz}
\begin{split}
\sigma_g^n(a\otimes f)\circ S_{n,z}\otimes E_{z,z}^{n+1}=& 
\sigma_g^{n+1}\left(a\otimes f\circ S_{n,z}\otimes E_{z,z}^{n+1}\right) ,
 \\
\beta_g^n(b\otimes h)\circ Q^n_z\otimes E_{z,z}^{n+1}=& \beta_g^{n+1}\left(b\otimes h\circ Q^n_z\otimes E_{z,z}^{n+1}\right).
\end{split}
\end{equation}
The second equation is trivial. For the first equation, since 
$\pi_{n,n+1}(k+g)=\pi_{n,n+1}(k)+g$, we get
\begin{align*}
&\quad\,\,\sigma_g^n(a\otimes f)\circ S_{n,z}(x,k)\otimes E_{z,z}^{n+1} 
\\&=\sigma_g^n(a\otimes f)(P_z^{(n)}(x), \pi_{n+1,n}(k))\otimes E_{z,z}^{n+1}\\
&= u_g^n(a\otimes f(P_z^n(x),\pi_{n+1,n}(k)-g))(u_g^n)^*\otimes E_{z,z}^{n+1}\\
&\numeq{\ref{relation_ug_Ez}} u_g^n(a\otimes f(P_z^n(x),\pi_{n+1,n}(k)-g))(u_g^n)^*\otimes w_g^{n+1}E_{z,z}^{n+1}(w_g^{n+1})^*
\\&= u_g^{n+1}(a\otimes f\otimes E_{z,z}^{n+1}(P_z^n(x),\pi_{n+1,n}(k)-g))(u_g^{n+1})^*
\\
&=\sigma_g^{n+1}\left(a\otimes f\circ S_{n,z}\otimes E_{z,z}^{n+1}\right)(x,k)
\end{align*}
Now, let us prove (\ref{action_well_defined}):
{\allowdisplaybreaks
\begin{align*}
&\Gamma_{n+1,n}\circ\alpha_g^{n}(a\otimes f,b\otimes h)
=\Gamma_{n+1,n}\left(\sigma_g^n(a\otimes f), \beta_g^n(b\otimes h)\right) \\
=& \Bigg(\sum_{z\in \bZ_{2^n}^d} \sigma_g^n(a\otimes 
f)(x_n,z)\otimes E^{n+1}_{z,z} + \beta^n_g(b\otimes h)(y_n)\otimes 
E^{n+1}_{\bigstar,\bigstar}
\\ & \quad +\sum_{z=1}^{d(n+1)} \sigma^n_g(a\otimes f)\circ S_{n,z}\otimes 
E^{n+1}_{z,z},\\ 
 &  \sum_{z\in \bZ_{2^n}^d} \sigma^n_g(a\otimes f)(x_n,z)\otimes 
 E^{n+1}_{z,z}+ \beta_g^n(b\otimes h)(y_n)\otimes E^{n+1}_{\bigstar,\bigstar}\\ 
 &\quad + \sum_{z=d'(n+1)+1}^{d(n+1)} \beta^n_g(b\otimes h)(y_n)\otimes E^{n+1}_{z,z} + 
 \sum_{z=1}^{d'(n+1)}\beta^n_g(b\otimes h)\circ Q_z^n\otimes 
 E^{n+1}_{z,z}\Bigg)\\
 = & \Bigg(\sum_{z\in \bZ_{2^n}^d} u_g^n(a\otimes 
 f)(x_n,z-g)(u_g^n)^*\otimes E^{n+1}_{z,z} + u_g^n(b\otimes 
 h)(y_n)(u_g^n)^*\otimes E^{n+1}_{\bigstar,\bigstar}
\\& \quad + \sum_{z=1}^{d(n+1)} \sigma_g^{n+1}(a\otimes f\circ S_{n,z}\otimes 
E^{n+1}_{z,z}),\\ 
 & \sum_{z\in \bZ_{2^n}^d} u_g^n(a\otimes 
 f)(x_n,z-g)(u_g^n)^*\otimes E^{n+1}_{z,z}+ u_g^n(b\otimes 
 h)(y_n)(u_g^n)^*\otimes E^{n+1}_{\bigstar,\bigstar}\\ & \quad  
 +\sum_{z=d'(n+1)+1}^{d(n+1)} 
 u^n_g(b\otimes h)(y_n)(u_g^n)^*\otimes E^{n+1}_{z,z} + 
 \sum_{z=1}^{d'(n+1)}\beta^{n+1}_g(b\otimes h\circ Q_z^n\otimes 
 E^{n+1}_{z,z})\Bigg)\\
 =&  \Bigg(\sum_{z\in \bZ_{2^n}^d} u_g^n a (u_g^n)^*\otimes 
 f(x_n,z)\otimes E^{n+1}_{z+g,z+g} + u_g^n b(u_g^n)^*\otimes h(y_n)\otimes 
 E^{n+1}_{\bigstar,\bigstar}
\\& \quad  + \sum_{z=1}^{d(n+1)} \sigma_g^{n+1}(a\otimes f\circ S_{n,z}\otimes 
E^{n+1}_{z,z}),\\ 
 & \sum_{z\in \bZ_{2^n}^d} u_g^n a (u_g^n)^*\otimes 
 f(x_n,z)\otimes E^{n+1}_{z+g,z+g}+ u_g^n b (u_g^n)^*\otimes h(y_n)\otimes 
 E^{n+1}_{\bigstar,\bigstar}\\ &\quad  +\sum_{z=d'(n+1)+1}^{d(n+1)} u^n_g b 
 (u_g^n)^*\otimes 
 h(y_n)\otimes E^{n+1}_{z,z} + \sum_{z=1}^{d'(n+1)}\beta^{n+1}_g(b\otimes 
 h\circ Q_z^n\otimes E^{n+1}_{z,z})\Bigg) \\ 
 =&  \Bigg(\sum_{z\in \left(\bZ_{2^n}\right)^d} u_g^{n+1}(a \otimes 
 f(x_n,z)\otimes E^{n+1}_{z,z})(u_g^{n+1})^* + u_g^{n+1} (b\otimes 
 h(y_n)\otimes E^{n+1}_{\bigstar,\bigstar})(u_g^{n+1})^*
\\&\quad  + \sum_{z=1}^{d(n+1)} \sigma_g^{n+1}(a\otimes f\circ S_{n,z}\otimes 
E^{n+1}_{z,z}),\\ 
 & \sum_{z\in \left(\bZ_{2^n}\right)^d} u_g^{n+1}(a \otimes f(x_n,z)\otimes 
 E^{n+1}_{z,z})(u_g^{n+1})^*+ u_g^{n+1} (b\otimes h(y_n)\otimes 
 E^{n+1}_{\bigstar,\bigstar})(u_g^{n+1})^*\\ & \quad +\sum_{z=d'(n+1)+1}^{d(n+1)} 
 u^{n+1}_g (b 
 \otimes h(y_n)\otimes E^{n+1}_{z,z})(u_g^{n+1})^* + 
 \sum_{z=1}^{d'(n+1)}\beta^{n+1}_g(b\otimes h\circ Q_z^n\otimes 
 E^{n+1}_{z,z})\Bigg).\\
 =& \Bigg(\sigma_g^{n+1}\Big(\sum_{z\in \left(\bZ_{2^n}\right)^d}a \otimes 
f(x_n,z)\otimes E^{n+1}_{z,z}+ b\otimes h(y_n)\otimes 
E^{n+1}_{\bigstar,\bigstar}
\\&\quad + \sum_{z=1}^{d(n+1)}a\otimes f\circ S_{n,z}\otimes E^{n+1}_{z,z}\Big),\\ 
 & \beta_{n+1}^g\Big(\sum_{z\in \left(\bZ_{2^n}\right)^d} a \otimes 
 f(x_n,z)\otimes E^{n+1}_{z,z}+ b\otimes h(y_n)\otimes 
 E^{n+1}_{\bigstar,\bigstar}\\ & 
 \quad +\sum_{z=d'(n+1)+1}^{d(n+1)} b \otimes h(y_n)\otimes E^{n+1}_{z,z} + 
 \sum_{z=1}^{d'(n+1)}b\otimes h\circ Q_z^n\otimes E^{n+1}_{z,z}\Big)\Bigg).\\
  = & \alpha_g^{n+1}\Bigg\{\Bigg(\sum_{z\in \left(\bZ_{2^n}\right)^d}a \otimes 
 f(x_n,z)\otimes E^{n+1}_{z,z}+ b\otimes h(y_n)\otimes 
 E^{n+1}_{\bigstar,\bigstar}
\\& \quad + \sum_{z=1}^{d(n+1)}a\otimes f\circ S_{n,z}\otimes E^{n+1}_{z,z},\\ 
 & \sum_{z\in \left(\bZ_{2^n}\right)^d} a \otimes f(x_n,z)\otimes 
 E^{n+1}_{z,z}+ b\otimes h(y_n)\otimes E^{n+1}_{\bigstar,\bigstar}\\ & 
\quad +\sum_{z=d'(n+1)+1}^{d(n+1)} b \otimes h(y_n)\otimes E^{n+1}_{z,z} + 
 \sum_{z=1}^{d'(n+1)}b\otimes h\circ Q_z^n\otimes E^{n+1}_{z,z}\Bigg)\Bigg\}.\\
 & \numeq{\ref{def_gamma_n}} \alpha_g^{n+1}\circ \Gamma_{n+1,n}(a\otimes f,b\otimes h).
 \end{align*}
}

Thus, this sequence of automorphisms is consistent with the connecting maps, 
and hence defines an automorphism  
$\alpha:A\rightarrow A$ on the direct limit.

For (\ref{L_af_proper.c_zd}), suppose there exists $g\in \bZ^d \smallsetminus 
\{0\}$ such that $\alpha_g$ is inner. 
Let $v$ be a unitary such that $\alpha_g = \Ad(v)$. For $n \in \bN$, let 
\[
p_n \in C_n \cong  C(X_n\times \bZ_{2^n}^d)\otimes M_{r(n)} \subseteq A_n
\]
 be the 
central projection given by
\[
p_n(x,k)=\delta_{0,k}\cdot 1_{M_{r(n)}}.
\] 
We identify $p_n$ with its image in $A$ in the direct limit. Because $p_n \in 
Z(A_n)$ for all $n$, we have $\| vp_nv^* - p_n \| \underset{n \to 
\infty}{\longrightarrow}0$. 
On the other hand, if $g \not \in 2^n\bZ^d$, then $\alpha_g(p_n)$ is orthogonal 
to $p_n$, so $\| \alpha_g(p_n) - p_n \| = 1$ for all but finitely many $n$'s. 
This is a contradiction.

Part (\ref{L_af_proper.d_zd}) follows from Theorem 3.1 of \cite{Ks1}, using the 
fact that $A$ is simple and $\alpha$ is pointwise outer.

\end{proof}
\section{Computation of the radius of comparison of $A$}\label{sec:rc(A)}

In this section, we compute the radius of comparison of the AH algebra 
constructed in Section~\ref{sec:construction_algebra}.

\begin{ntn}\label{notation}
Let $L$ be the tautological line bundle over $S^2\cong \mathbb{C}\mathbb{P}^1$ 
and let $b\in C(S^2,M_2)$ denote the Bott projection, that is, the projection 
onto the section space of $L$.  For $n\geq 0$ we define a projection \[p_n\in 
M_{2h_ns(n)}\left(C\left( (S^2)^{h_ns(n)},M_{r(n)}\right)\right)\] as follows. 
For $i\in\{1,2,\hdots h_ns(n)\}$, let 
$R_i:\left(S^2\right)^{h_ns(n)}\rightarrow S^2$ be the $i$-th coordinate 
projection. Let $e\in M_2(\mathbb{C})$ be a trivial rank one projection. 

Define 
\begin{eqnarray*}p_n:=\text{diag}(b\circ R_1,b\circ R_2,\hdots, b\circ R_{h_ns(n)}, 
\underbrace{e,e,e,\hdots e}_{h_ns(n)\left(r(n)-1\right)}).\end{eqnarray*} 

Then $p_n$ is a 
direct sum of projection of rank $h_ns(n)$ onto the section space of the 
Cartesian product $L^{\times h_ns(n)}$ and a trivial bundle of rank 
$h_ns(n)r(n)-h_ns(n)$. Let \[q_n\in M_{2h_ns(n)}\left(C\left( 
(S^2)^{h'_ns(n)},M_{r(n)}\right)\right)\] be a constant projection of rank 
$h_ns(n)r(n)$. 

Denote by
$\overline{p}_n$ the image of $1_{C(\bZ_{2^n}^d)} \otimes p_n$ under the 
identification
\begin{eqnarray*}
 C(\bZ_{2^n}^d) \otimes 
M_{2h_ns(n)}\left(C\left( (S^2)^{h_ns(n)},M_{r(n)}\right)\right)\cong 
M_{2h_ns(n)}\left(C\left( (S^2)^{h_ns(n)}\times 
\mathbb{Z}_{2^n}^d,M_{r(n)}\right)\right).
 \end{eqnarray*}
 
Then for $m\geq n\geq 0$ we have 
\begin{eqnarray*}\left(\text{id}_{M_{2h_ns(n)}}\otimes 
\Gamma_{m,n}\right)(\overline{p}_n,q_n)=(\overline{p}_{m,n},q_{m,n}).
\end{eqnarray*}

Write 
$\overline{p}_{m,n}=\overline{p}_{m,n}^{(1)}\oplus \overline{p}_{m,n}^{(2)}$, 
where 
on 
each connected component of $X_m\times\mathbb{Z}_{2^m}^d$, 
$\overline{p}_{m,n}^{(1)}$ is a projection of rank $h_ns(m)$ onto the section 
space 
of the Cartesian product $L^{\times h_ns(m)}$ and $\overline{p}_{m,n}^{(2)}$ is 
a 
trivial bundle of rank $h_ns(n)r(m)-h_ns(m)$. 
\end{ntn}


We recall the following fact.

\begin{rmk}\label{rem}
The Cartesian product $L^{\times k}$ does not embed in a trivial bundle over $(S^2)^k$ of rank less than $2k$ (see \cite[Lemma 1.9]{HP19}).
\end{rmk}

The following lemma is a fairly standard lower bound for the 
rank of a trivial bundle which contains a direct product of copies of the 
tautological line bundle; see for instance \cite[Corollaries 6.13 and 
6.20]{AGP19} for similar estimates.

 \begin{lem}\label{lem_sec5_rankbound_}
	Let $n\in\{0,1,2,\ldots\}$ be fixed. Let $m\geq n\in\bZ$. Let 
	$e=(e_1,e_2)$ 
	be a 
	projection in $M_{\infty}(A_m)$ such that $e_1$ is trivial. If there exists 
	$x\in M_\infty(A_m)$ such that \[\Vert 
	xex^*-(\overline{p}_{m,n},q_{m,n})\Vert<\frac{1}{2},\] then 
	\[\text{rank}(e_1)\geq h_ns(n)r(m)+h_ns(m).\]
\end{lem} 

\begin{proof}
	The assumption	
	 $\Vert xex^*-(\overline{p}_{m,n}, q_{m,n})\Vert<\frac{1}{2}$ implies that
	\begin{equation}
		\begin{split}
			\Vert \overline{p}_{m,n}\left(xex^*|_{X_m\times\bZ_{2^m}^d}\right)\overline{p}_{m,n}-\overline{p}_{m,n}\Vert<\frac{1}{2}.
		\end{split}
	\end{equation}
	
	Recall that $X_m \cong (S^2)^{h_m s(m)} \cong   (S^2)^{h_n s(m)} \times  
	(S^2)^{(h_m -h_n) s(m)}$. Fix some point $z_0 \in (S^2)^{(h_m -h_n) 
	s(m)}$. 
	Set 
	\[
	X_{m,n} = (S^2)^{h_ns(m)}\times \{z_0\} \subseteq X_m .
	\] 
	
	Then
	\begin{equation}
		\begin{split}
			\Vert \overline{p}_{m,n}|_ {X_{m,n}\times \bZ_{2^m}^d}\left(xex^*|_{X_{m,n}\times\bZ_{2^m}^d}\right)\overline{p}_{m,n}|_{X_{m,n}\times \bZ_{2^m}^d}-\overline{p}_{m,n}|_{X_{m,n}\times \bZ_{2^m}^d}\Vert<\frac{1}{2}.
		\end{split}
	\end{equation}
	
	Since $\overline{p}_{m,n}$ and $e$ are projections, we get
	\[\overline{p}_{m,n}|_{X_{m,n}\times \bZ_{2^m}^d}\lesssim 
	e|_{X_{m,n}\times\bZ_{2^m}^d}=e_1|_{X_{m,n}\times\bZ_{2^m}^d}.\] 
	Therefore, 
	there exists a projection 
	\[q\in M_{\infty}(C(X_{m,n}\times \bZ_{2^m}^d,M_{r(m)}))\] such that
	\[
	\overline{p}_{m,n}|_{X_{m,n}\times \bZ_{2^m}^d}+q\approx 
	e_1|_{X_{m,n}\times \bZ_{2^m}^d}.
	\]
	Let 
	\[e_0\in M_{\infty}(C(X_{m,n}\times \bZ_{2^m}^d,M_{r(m)}))\] be a trivial 
	projection whose rank is 
	\[\text{rank}(e_1|_{X_{m,n}\times 
	\bZ_{2^m}^d})-\text{rank}(\overline{p}_{m,n}^{(2)}|_{X_{m,n}\times\bZ_{2^m}^d}),\]
	 and such that $e_0\perp 
	\overline{p}_{m,n}^{(2)}|_{X_{m,n}\times\bZ_{2^m}^d}$, then 
	\[\overline{p}_{m,n}^{(2)}|_{X_{m,n}\times\bZ_{2^m}^d}+e_0\approx 
	e_1|_{X_{m,n}\times \bZ_{2^m}^d}.\]
	 
	 Hence 
	\begin{equation}
		\begin{split}
			\overline{p}_{m,n}^{(1)}|_{X_{m,n}\times\bZ_{2^m}^d}+
			\overline{p}_{m,n}^{(2)}|_{X_{m,n}\times\bZ_{2^m}^d}+q\approx
			 \overline{p}_{m,n}^{(2)}|_{X_{m,n}\times\bZ_{2^m}^d}+e_0.
		\end{split}
	\end{equation} 
	
	Now, 
	$\text{rank}(\overline{p}_{m,n}^{(1)}|_{X_{m,n}\times\bZ_{2^m}^d}+q) 
	\geq h_ns(m)$. 
	Thus, by \cite[Theorem 9.1.5]{HusB}, we get that 
	$\overline{p}_{m,n}^{(1)}|_{X_{m,n}\times\bZ_{2^m}^d}\lessapprox e_0$, and 
	hence as a consequence of Remark  \ref{rem}, we have $\text{rank}(e_0)\geq 
	2h_ns(m)$. Since 
	\[\overline{p}_{m,n}^{(2)}|_{X_{m,n}\times\bZ_{2^m}^d}+e_0\approx 
	e_1|_{X_{m,n}\times \bZ_{2^m}^d},\] we get 
	\begin{align*}
	\begin{split}
	\text{rank}(e_1) =\text{rank}(e_1|_{X_{m,n}\times \bZ_{2^m}^d})&\geq 
	2h_ns(m)+h_ns(n)r(m)-h_ns(m)\\
	&=h_ns(n)r(m)+h_ns(m).
	\end{split}
	\end{align*}
\end{proof}

\begin{prp}\label{rad_algebra}
Assume the notation and choices in Construction~\ref{OurConstruction_zd}.
Then  $\rc (A)=r$.
\end{prp}
\begin{proof}
We first show that $\rc(A)\leq r$. Using Proposition \ref{prop_roc} and \cite[Corollary 1.2]{EN13}, we see that, 
for $n\in\{0,1,2,\ldots\}$
\[
\rc(C_n)=\frac{1}{r(n)} \rc\left(C(X_n\times \bZ_{2^n}^d)\right)\leq 
\frac{\dim(X_n)}{2r(n)}=\frac{h_ns(n)}{r(n)},
\]
and
\[
\rc(B_n)=\frac{1}{r(n)}\rc \left(C(Y_n)\right)\leq 
\frac{\dim(Y_n)}{2r(n)}=\frac{h'_n s'(n)}{r(n)} .
\]

Hence,
\begin{equation*}
\begin{split}
\rc (A_n)=\max\{\rc(C_n),\rc (B_n)\}\leq\frac{h_n s(n)}{r(n)}.
\end{split}
\end{equation*}

Since $A$ and all the $A_n$, for $n\in\bN$, are residually stably finite, it 
follows from  
\cite[Proposition 3.2.3 and 3.2.4(iii)]{blackadar2012algebraic} (see also \cite[Proposition 2.13]{AA20}) that
\begin{equation}
\begin{split}
\rc (A) &\leq \lim_{n\to\infty}\inf\rc (A_n)
\\
&\leq\lim_{n\to\infty}\frac{h_n s(n)}{r(n)}\\
&=\lim_{n\to\infty}h_n\lim_{n\to\infty}\prod_{j=1}^{n}\frac{d(j)}{d(j)+1+2^{jd-d}}=r
\end{split}
\end{equation}

We now show that $r\leq \rc (A)$. We first consider the case when $0< 
r<\infty$. As in Notation~\ref{notation}, let  
 $p_0\in C(\left(S^2\right)^{h_0}, M_{2{h_0}})$ be the projection of rank $h_0$ 
 whose range is the section 
space of $L^{\times h_0}$. Then 
\[\left(\overline{p}_{n,0}, 
q_{n,0}\right)=\left(\text{id}_{M_{2h_0}}\otimes\Gamma_{n,0}\right)(p_0,q_0),\] and on each of the connected component of $X_n\times 
\bZ_{2^n}^d$, $\overline{p}_{n,0}$ is a direct sum of a projection of rank 
$h_0s(n)$ whose range is the section space of $L^{\times h_0s(n)}$ and a 
trivial projection 
of rank $h_0r(n)-h_0s(n)$.  By Lemma \ref{lem_sec5_rankbound_}, it follows
that for $n\in\{0,1,2,\ldots\}$, if $e=(e_1,e_2)$ is a projection in 
$M_{\infty}(A_n)$ such that $e_1$ is trivial and $x\in M_{\infty}(A_n)$ is such 
that \[\Vert xex^*-(\overline{p}_{n,0},q_{n,0})\Vert<\frac{1}{2},\] then $\text{rank}(e_1)\geq h_0r(n)+h_0s(n)$. Now, 
we fix $0<\rho<r$ and show that $A$ cannot have $\rho$-comparison.  Note that 
\[0<\frac{\rho}{h_0}<\frac{r}{h_0}=\kappa.\] 

Choose $n\in\bN$ such that 
\[\frac{1}{h_0r(n)}<\kappa-\frac{\rho}{h_0}.\] 

Choose $M\in\bN$ such that 
\begin{equation}\label{equ_inequality_sec5_}
	\frac{\rho}{h_0}+1<\frac{M}{h_0r(n)}<\kappa+1.
\end{equation}

Let $e=(e_1,e_2)\in M_{\infty}(A_n)$ be a trivial projection of rank $M$. By a
slight abuse of notation, we use $\Gamma_{m,n}$ to denote the amplified map 
from $M_\infty(A_n)$ to $M_\infty(A_m)$. For $m>n$, the rank of 
$\Gamma_{m,n}(e)$ is $M\cdot \frac{r(m)}{r(n)}$, which by the choice of $M$ is 
strictly less than $h_0r(m)+h_0s(m)$. Now, for any trace $\tau$ on $A_m$ (and 
thus for any trace on $A$), we have
\[\tau\left(\overline{p}_{m,0}, 
q_{m,0}\right)=\frac{h_0r(m)}{r(m)}=h_0.\]

 Using (\ref{equ_inequality_sec5_}), we get
\begin{equation*}
	\begin{split}
		d_\tau\left(\Gamma_{m,n}(e)\right)=\tau\left(\Gamma_{m,n}(e)\right)&
		=\frac{1}{r(m)}\cdot M\cdot \frac{r(m)}{r(n)}\\
		&>h_0 + \rho\\
		&=\tau\left(\overline{p}_{m,0}, 
		q_{m,0}\right)+\rho\\
		&=d_\tau\left(\overline{p}_{m,0}, 
		q_{m,0}\right)+\rho.
	\end{split}
\end{equation*}

Now, if $\Gamma_{\infty,0}(p_0,q_0)\lesssim\Gamma_{\infty,n}(e)$, then, in particular, there exists some $m>n$ and $x\in M_{\infty}(A_m)$ such that 
\[\Vert x\Gamma_{m,n}(e)x^*-\left(\overline{p}_{m,0}, 
q_{m,0}\right)\Vert<\frac{1}{2}.\]

 This implies that $\text{rank}( \Gamma_{m,n}(e) ) \geq h_0r(m)+h_0s(m)$, which 
 is a contradiction. Hence, $\rc (A)\geq r$.

It remains to consider the case $r=\infty$. For 
this, we show that $A$ cannot have $\rho$-comparison for any $\rho \in 
(0,\infty)$. 
Set $c = \inf_{m \geq 0}\frac{s(m)}{r(m)} >0$. Recall that 
$\lim_{n\rightarrow\infty}h_n=\infty$. Fix $n\in\bN$ such that
\begin{equation}
		 h_n  > \rho/c \quad \text{and} \quad 
		0< \frac{1}{r(n)}   < c \cdot h_n - \rho.
\end{equation}

Choose $M\in\bN$ such that 
\begin{equation}\label{equ_inequality_sec6_4}
	\rho+h_ns(n)<\frac{M}{r(n)}<c \cdot h_n + h_n s(n).
\end{equation}

Let $e=(e_1,e_2)\in M_{\infty}(A_n)$ be a trivial projection of rank $M$. By 
slight abuse of notation, we use $\Gamma_{m,n}$ to denote the amplified map 
from $M_\infty(A_n)$ to $M_\infty(A_m)$. For $m\geq n$, the rank of 
$\Gamma_{m,n}(e)$ is $M\cdot\frac{r(m)}{r(n)}$, which by the choice of $M$ is 
strictly less than $r(m)c h_n+ h_nr(m)s(n)$. Now
 \[h_n c r(m)\leq h_n\frac{s(m)}{r(m)}r(m)=h_ns(m).\] 
 Hence, for $m\geq n$, the
 rank of $\Gamma_{m,n}(e)$ is less than $h_ns(m)+h_ns(n)r(m)$.  
 
 Now, for any 
 trace $\tau$ on 
 $A_m$ (and thus for any trace on $A$) \[\tau(\left(\overline{p}_{m,n}, 
 q_{m,n}\right))=\frac{h_ns(n)r(m)}{r(m)}=h_ns(n).\] Moreover, by Inequality 
 \ref{equ_inequality_sec6_}, we get
\begin{equation*}
	\begin{split}
		\tau\left(\Gamma_{m,n}(e)\right)=\frac{1}{r(m)}\cdot M\cdot \frac{r(m)}{r(n)}>h_ns(n)+\rho=\tau(\left(\overline{p}_{m,n}, q_{m,n}\right))+\rho.
	\end{split}
\end{equation*}

Now, if $\Gamma_{\infty,n}(\left(\overline{p}_{n}, q_{n}\right))\lesssim\Gamma_{\infty,n}(e)$, then, in particular, there exists some $m>n$ and $x\in M_{\infty}(A_m)$ such that 

\[\Vert x\Gamma_{m,n}(e)x^*-\left(\overline{p}_{m,n}, q_{m,n}\right)\Vert<\frac{1}{2}.\] 

By Lemma~\ref{lem_sec5_rankbound_},  
 \[\text{rank}(\Gamma_{m,n}(e))\geq h_ns(n)r(m)+h_ns(m),\] which is a 
 contradiction. Hence, $\rc (A)\geq \rho$, as claimed.

\end{proof}

\section{Computation of the radius of comparison of $A\rtimes_{\alpha}\bZ^d$}\label{sec:rc-crossed-product}

In this section, we compute the radius of comparison of 
$A\rtimes_{\alpha}\bZ^d$, for the AH algebra $A$ and the action $\alpha$ as described in Section~\ref{sec:construction_algebra}.

Let $z\in C(\mathbb{T})$ be the function $z(\zeta) = \zeta$ for
all $\zeta$. Set 
\[
z_n= \underbrace{\begin{pmatrix}
		0 & 0 & \cdots & 0 & 0 & z \\
		1 & 0 & \cdots & 0 & 0 & 0 \\
		0 & 1 & \cdots & 0 & 0 & 0  \\
		\vdots  & \vdots  & \ddots & \vdots & \vdots & \vdots \\
		0 & 0 & \cdots & 1 & 0 & 0\\
		0 & 0 & \cdots & 0 & 1 & 0 
\end{pmatrix}}_{2^n\times 2^n} \in M_{2^n} \otimes C(\mathbb{T}) 
\]
and for $g = (k_1,k_2,\ldots,k_d) \in \bZ^d$, define $z_n^g = z_n^{k_1} \otimes 
z_n^{k_2} \otimes \cdots \otimes z_n^{k_1}$; we identify $z$ with $z_0$. 

For $n\in\bN $, let $\widetilde{C_n}=C_n\rtimes_{\sigma_n}\bZ^d$, and 
$\widetilde{B_n}=B_n\rtimes_{\beta_n}\bZ^d$. Then as in Examples 9.6.4 and 
9.6.11, and Exercises 9.6.26 and 9.6.27 of \cite{GKPT18}, we conclude that 
\begin{equation}
\begin{split}
\widetilde{C_n}\cong  M_{r(n)}\left(C(X_n)\right)&\otimes M_{2^{nd}}\otimes C(\mathbb{T}^d),\,\text{and}\\
\widetilde{B_n}\cong  M_{r(n)}\left(C(Y_n)\right)&\otimes C(\mathbb{T}^d),
\end{split}
\end{equation}
via the following isomorphisms. Use 
$(\lambda_g)_{g \in \bZ^d}$ for the canonical unitary 
$\widetilde{C_n} \oplus \widetilde{B_n}$. 
Define
\[
\psi_n \colon \widetilde{C_n} \oplus \widetilde{B_n} \to    
M_{r(n)}\otimes C(X_n) \otimes M_{2^{nd}}\otimes C(\mathbb{T}^d)
\oplus  M_{r(n)}\otimes C(Y_n) \otimes C(\mathbb{T}^d)
\]
as follows. Identify $M_{2^{nd}}$ with $B(\ell^2(\bZ_{2^n}^d))$, and identify 
$C(\bZ_{2^n}^d)$ with the diagonal operators in $B(\ell^2(\bZ_{2^n}^d)$. Let 
$\iota_n \colon C(\bZ_{2^n}^d) \to B(\ell^2(\bZ_{2^n}^d))$ be this inclusion.
For $a \in C_n \cong M_{r(n)} \otimes C(X_n) \otimes C(\bZ_{2^n}^d)$ and $b \in 
B_n \cong M_{r(n)} \otimes C(Y_n)$, we have \[\psi_n (a \oplus b) = 
\id_{M_{r(n)} \otimes C(X_n)} \otimes \iota_n (a) \oplus b .\]
For $g \in \bZ^d$, we have 
\[
\psi_n (\lambda_g) = ( u_g^n \otimes 1_{X_n}\otimes 1_{M_{2^{nd}}} \otimes 
z_n^g ) \oplus ( u_g^n \otimes 1_{Y_n} \otimes z_n^g ) .
\]

Let $\Gamma_{n+1,n}\rtimes \bZ^d$ be the corresponding induced map of 
$\Gamma_{n+1,n}$ at the level of crossed products, and let 
$\tilde{\Gamma}_{n+1,n}$ 
be the unital map such that the diagram below commutes

\xymatrix{
*\txt{$\widetilde{C_n}\oplus \widetilde{B_n}$}
\ar[d]^-{\psi_n}\ar[rr]^-{\txt{$\Gamma_{n+1,n}\rtimes\mathbb{Z}^d$}} 
&&*\txt{$\,\widetilde{C_{n+1}}\oplus 
\widetilde{B_{n+1}}$}\ar[d]^-{\psi_{n+1}}\\ 
*\txt{$M_{r(n)}\left(C(X_n)\right)\otimes M_{2^{nd}}\otimes C(\mathbb{T}^d)$\\ $\oplus$\\ $M_{r(n)}\left(C(Y_n)\right)\otimes C(\mathbb{T}^d)$}
\ar[rr]^-{\txt{$\tilde{\Gamma}_{n+1,n}$}}
&&*\txt{$\,M_{r(n+1)}\left(C(X_{n+1})\right)\otimes M_{2^{dn+d}}\otimes C(\mathbb{T}^d)$\\ $\oplus$\\ $M_{r(n+1)}\left(C(Y_{n+1})\right)\otimes C(\mathbb{T}^d)$.}}

\vspace*{0.3cm}

Then,
\[
A\rtimes_{\alpha}\bZ^d 
\cong\displaystyle{\lim_{\longrightarrow}} 
\left\{M_{r(n)} \left(C(X_n)\right)\otimes M_{2^{dn}}\otimes 
C(\mathbb{T}^d)\oplus M_{r(n)}\left(C(Y_n)\right)\otimes  
C(\mathbb{T}^d)\right\}.
\]

%



For $f\in M_{r(n)}\left (C(X_n)\right)$ and for $g\in 
M_{r(n)}\left(C(Y_n)\right)$, 
we have

\begin{equation}\label{eq_gamma_crossed}
 \begin{split}
 \tilde{\Gamma}_{n+1,n}(f\otimes 1_{2^{nd}} \otimes 1_{C(\mathbb{T}^d)} 
 ,g\otimes 
 1_{C(\mathbb{T}^d)})=&\Bigg(\bigg(\sum_{z\in \bZ_{2^n}^d}f(x_n)\otimes 
 E_{z,z}^{n+1}+ g(y_n)\otimes E_{\bigstar,\bigstar}^{n+1} + \\ 
 &\sum_{z=1}^{d(n+1)}f\circ P^n_z\otimes E_{z,z}^{n+1}\bigg)\otimes 
 1_{2^{nd+d}},\\
 & \bigg(\sum_{z\in \bZ_{2^n}^d}f(x_n)\otimes E_{z,z}^{n+1}+\sum_{z=d'(n+1)+1}^{d(n+1)} g(y_n)\otimes E_{z,z}^{n+1}+\\
 &  g(y_n)\otimes E_{\bigstar,\bigstar}^{n+1}+\sum_{z=1}^{d'(n+1)}g\circ 
 Q_{z}^n\otimes E_{z,z}^{n+1} \bigg)\otimes 1_{C(\mathbb{T}^d)}\Bigg).
 \end{split}
 \end{equation}

 Before we prove the main result of this section, we fix some notation.
 
 \begin{ntn}\label{rem_sec4}
For $n\geq 0$, define a projection \[\tilde{q}_n\in M_{2h'_ns'(n)}\left(C\left( 
(S^2)^{h'_ns'(n)},M_{r(n)} \right)\right)\otimes C(\mathbb{T}^d)\] as follows: 
for $i\in\{1,2,\hdots h_n's'(n)\}$, let 
$T_i:\left(S^2\right)^{h'_ns'(n)}\rightarrow S^2$ be the $i$-th coordinate 
projection. Let $e\in M_2(\mathbb{C})$ be a trivial rank one projection. 

Define 
\[\tilde{q}_n:=\text{diag}(b\circ T_1,b\circ T_2,\hdots, b\circ T_{h'_ns'(n)}, 
\underbrace{e,e,e,\hdots e}_{h'_ns'(n)\left(r(n)-1\right)})\otimes 
1_{C(\mathbb{T}^d)}.\] 

Then, $\tilde{q}_n|_{(S^2)^{h'_ns'(n)}\times \{1\}}$ is 
a direct sum of a projection of rank $h'_ns'(n)$ onto the section space of the 
Cartesian product $L^{\times h'_ns'(n)}$ and a trivial bundle of rank 
$h'_ns'(n)r(n)-h'_ns'(n)$. Set
\[\tilde{p}_n\in M_{2h'_ns'(n)}\left(C\left( 
(S^2)^{h_ns(n)},M_{r(n)}\right)\right)\otimes M_{2^{dn}}\otimes 
C(\mathbb{T}^d)\] 
to be 
\[\tilde{p}_n:=\text{diag}(\underbrace{e,e,e,\hdots,e}_{h'_ns'(n)r(n)})\otimes1_{2^{dn}}\otimes
 1_{C(\mathbb{T}^d)}.\]
 
For $m\geq n\geq 0$, set
 \[\left(\text{id}_{M_{2h'_ns'(n)}}\otimes 
 \tilde{\Gamma}_{m,n}\right)(\tilde{p}_n,\tilde{q}_n)=(\tilde{p}_{m,n},\tilde{q}_{m,n}).\]
  (For $n\geq 0$, $\tilde{p}_{n,n}=\tilde{p}_n$ and 
 $\tilde{q}_{n,n}=\tilde{q}_n$.) Using (\ref{eq_gamma_crossed}), we have a 
 decomposition
\[\tilde{q}_{m,n}=\left(\tilde{q}_{m,n}^{(1)}\oplus 
\tilde{q}_{m,n}^{(2)}\right) \otimes 1_{C(\mathbb{T}^d)},\]
where $\tilde{q}_{m,n}^{(1)}$ is a 
projection of rank $h'_ns'(m)$ onto the section space of the Cartesian product 
$L^{\times h'_ns'(m)}$ and $\tilde{q}_{m,n}^{(2)}$ is a trivial bundle of rank 
$h'_ns'(n)r(m)-h'_ns'(m)$.
 \end{ntn}
 
 \begin{lem}\label{lem_sec4_rankbound}
 Let $n\in\{0,1,2,\ldots\}$ be fixed. Let $m\geq n$. Let $e=\left(e_1\otimes 1_{C(\mathbb{T}^d)}, 
 e_2\otimes 1_{C(\mathbb{T}^d)}\right)$ be a projection in 
 \[M_\infty\left(M_{2^{md}r(m)}C(X_m)\otimes C(\mathbb{T}^d)\oplus 
 M_{r(m)}C(Y_m)\otimes C(\mathbb{T}^d)\right)\] such that $e_2$ is trivial. If 
  there exists \[x\in M_\infty\left(M_{2^{md}r(n)}C(X_m)\otimes 
 C(\mathbb{T}^d)\oplus M_{r(m)}C(Y_m)\otimes C(\mathbb{T}^d)\right)\] such that 
 $\Vert xex^*-(\tilde{p}_{m,n},\tilde{q}_{m,n})\Vert<\frac{1}{2}$, then $\text{rank}(e_2)\geq h'_ns'(n)r(m)+h'_ns'(m).$
 \end{lem}
 
 \begin{proof}
 Note that $\Vert xex^*-(\tilde{p}_{m,n},\tilde{q}_{m,n})\Vert<\frac{1}{2}$ implies that 
 \begin{equation*}
 \left\Vert \tilde{q}_{m,n}\left(xex^*|_{ Y_n\times \mathbb{T}^d}\right)\tilde{q}_{m,n}-\tilde{q}_{m,n}\right\Vert<\frac{1}{2}.
 \end{equation*}
 
 From this we conclude that 
 \begin{equation}
 \left\Vert (\tilde{q}_{m,n}^{(1)}\oplus \tilde{q}_{m,n}^{(2)})\left(xex^*|_{ 
 Y_n\times \{1\}}\right)(\tilde{q}_{m,n}^{(1)}\oplus 
 \tilde{q}_{m,n}^{(2)})-(\tilde{q}_{m,n}^{(1)}\oplus 
 \tilde{q}_{m,n}^{(2)})\right\Vert<\frac{1}{2}.
 \end{equation}

Therefore
 \[(\tilde{q}_{m,n}^{(1)}\oplus \tilde{q}_{m,n}^{(2)})\lesssim e|_{Y_m\times 
 1}=e_2.\]
   
   Recall from Notation~\ref{rem_sec4} that $\tilde{q}_{m,n}^{(1)}$ is a 
   projection of rank $h'_ns'(m)$ whose range is the section space of 
   $L^{\times h'_ns'(m)}$ and $\tilde{q}_{m,n}^{(2)}$ is a trivial projection of 
   rank $h'_ns'(n)r(m)-h'_ns'(m)$. The rest of the proof is identical to that of
   Lemma \ref{lem_sec5_rankbound_}.
 \end{proof}
 \begin{prp}\label{rad_crossedproduct}
 Assume the notation and choices in Construction~\ref{OurConstruction_zd}.
Then $\rc (A\rtimes_{\alpha}\bZ^d)=r'$. 

\end{prp}
\begin{proof}
We first show that $\rc(A\rtimes_\alpha\bZ)\leq r'$. Using Proposition \ref{prop_roc} and \cite[Corollary 1.2]{EN13}, we see 
that, for $n\in\{0,1,2,\ldots\}$ 
\begin{equation*}
\begin{split}
\rc \left(M_{r(n)}\left(C(X_n)\right)\otimes M_{2^{nd}}\otimes C(\mathbb{T}^d) \right)&=\frac{1}{{2^{nd}}r(n)}\rc \left(C(X_n\times \mathbb{T}^d)\right)\\
&\leq \frac{\dim(X_n\times \mathbb{T}^d)}{{2^{dn+1}}r(n)}\\ &=\frac{2h_ns(n)+d}{{2^{dn+1}}r(n)}\\
&=\frac{h_ns(n)}{{2^{nd}}r(n)}+\frac{d}{{2^{dn+1}}r(n)},
\end{split}
\end{equation*}
and
\begin{equation*}
	\begin{split}
\rc \left(M_{r(n)}C(Y_n)\otimes C(\mathbb{T}^d)\right)&=\frac{1}{r(n)}\rc \left(C(Y_n\times \mathbb{T}^d)\right)\\
&\leq \frac{\dim(Y_n)+d}{2r(n)}\\ &=\frac{h'_ns'(n)}{r(n)}+\frac{d}{2r(n)}.
\end{split}
\end{equation*}

$A\rtimes_\alpha\bZ^d$ is stably finite and simple by Proposition~\ref{L_af_proper_zd}.  By Proposition 3.2.3 and 3.2.4(iii)  of \cite{blackadar2012algebraic}, since 
\[\lim_{n\to\infty}\frac{h_n}{2^{dn}}=0,\]
 we get
\begin{equation*}
\begin{split}
\rc (A\rtimes_\alpha\bZ^d)& \leq \lim_{n\to\infty}\inf\left(\max\left\{\frac{h_ns(n)}{{2^{dn}}r(n)}+\frac{d}{{2^{dn+1}}r(n)},\frac{h'_ns'(n)}{r(n)}+\frac{d}{2r(n)}\right\}\right)\\
&= \lim_{n\to\infty}\inf \left(\frac{h'_ns'(n)}{r(n)}+\frac{d}{2r(n)}\right)\\
&=\lim_{n\to\infty}\frac{h'_ns'(n)}{r(n)}=r'.
\end{split}
\end{equation*}

Now, let us show that $r'\leq \rc (A\rtimes_\alpha\bZ^d)$.  We first consider the case  $0<r' <\infty$.
Fix $0<\rho<r'$. We claim that $A\rtimes_\alpha\bZ^d$ does not have 
$\rho$-comparison.  Note that $0<\frac{\rho}{h'_0}<\kappa'$.
Choose $n\in\bN$ such that 
\[\frac{1}{h'_0r(n)}<\kappa'-\frac{\rho}{h'_0}.\]

 Choose $M\in\bN$ such 
that 
\begin{equation}\label{equ_inequality_sec5_part2_}
	\frac{\rho}{h'_0}+1<\frac{M}{h'_0r(n)}<\kappa'+1.
\end{equation}

 Denote 
 \[\widetilde{A_n}=M_{2^{dn}r(n)}C(X_n)\otimes C(\mathbb{T}^d)\oplus M_{r(n)}C(Y_n)\otimes C(\mathbb{T}^d).\]
 
  Let 
  \begin{equation}\label{proj_e}
  e=\left(e_1\otimes 1_{C(\mathbb{T}^d)}, e_2\otimes 1_{C(\mathbb{T}^d)}\right)
  \end{equation}
  be a projection in $M_\infty(\widetilde{A_n})$ such that $e_1$ is a trivial 
  projection of rank ${2^{nd}}M$ and $e_2$ is a trivial projection of rank $M$. 
  By a slight abuse of notation, we use $\tilde{\Gamma}_{m,n}$ to denote the 
  amplified map from $M_\infty (\widetilde{A_n})$ to $M_\infty 
  (\widetilde{A}_m)$. For $m>n$, let 
  \[\tilde{\Gamma}_{m,n}(e)=\left(\tilde{e}_1\otimes 1_{C(\mathbb{T}^d)}, \tilde{e}_2
  \otimes 1_{C(\mathbb{T}^d)}\right).\]
  
   Then, from Equation \ref{eq_gamma_crossed}, we can deduce that the rank of 
   $\tilde{e}_1$ is $\frac{Mr(m){2^{md}}}{r(n)}$, and the rank of $\tilde{e}_2$ 
   is $\frac{Mr(m)}{r(n)}$, which is strictly less than $h'_0s'(m)+h'_0r(m)$ by 
   Equation \ref{equ_inequality_sec5_part2_}. Now, if $\tau$ is a trace on 
   $\widetilde{A_m}$ then there exists $0<\lambda<1$ such that
 \[
 \tau(\tilde{\Gamma}_{m,n}(e))=\lambda\text{tr}_1(\tilde{e}_1)+ \left(1-\lambda\right)\text{tr}_2(\tilde{e}_2),
 \]
 where $\text{tr}_1$ is a normalized trace on $M_{{2^{dm}}r(m)}$ and $\text{tr}_2$ is a normalized trace on $M_{r(m)}$. 
 
  Hence,
  \begin{equation}
  	\begin{split}
  		\tau(\tilde{\Gamma}_{m,n}(e))&=\lambda\cdot \frac{1}{{2^{dm}}r(m)}\cdot \frac{Mr(m){2^{dm}}}{r(n)}+ (1-\lambda)\cdot \frac{1}{r(m)}\cdot\frac{Mr(m)}{r(n)}\\
  		&= \frac{M}{r(n)}>\rho +h'_0=\rho + \tau\left(\tilde{p}_{m,0},\tilde{q}_{m,0}\right).
  	\end{split}
  \end{equation}
  
On the other hand, if $\tilde{\Gamma}_{\infty,0}(\tilde{p}_0,\tilde{q}_0)\lesssim \tilde{\Gamma}_{\infty,n}(e)$ then, there exists some $m>n$ and $x\in M_{\infty}(\widetilde{A}_m)$ such that 
\[
\Vert x\tilde{\Gamma}_{m,n}(e)x^*-\left(\tilde{p}_{m,0},\tilde{q}_{m,0}\right) \Vert<\frac{1}{2}.
\]

By Lemma~\ref{lem_sec4_rankbound}, the above inequality implies that 

\[\text{rank}(\tilde{e}_2)\geq h'_0r(m)+h'_0s'(m),\] which is a contradiction by Equation \ref{equ_inequality_sec5_part2_}. 

Finally, we consider the case when $r'=r=\infty$ and show that $\rc(A\rtimes_\alpha\bZ^d)=\infty$. For this, we show that $\rc(A\rtimes_\alpha\bZ^d)$ cannot have $\rho$-comparison for any $\rho\in (0,\infty)$.
 Since $\lim_{n\rightarrow\infty}h'_n=\infty$, choose $n\in\bN$ such that 
\begin{equation}\label{equ_inequality_sec6_1_}
	\begin{split}
		h_n'&>\frac{\rho}{c},\\
		0<\frac{1}{r(n)}&<c\cdot h_n'-\rho.
	\end{split}
\end{equation}

Choose $M\in\bN$ such that 
\begin{equation}\label{equ_inequality_sec6_}
	\rho+h'_ns'(n)<\frac{M}{r(n)}<c\cdot h_n'+h'_ns'(n).
\end{equation}

Let $e=\left(e_1\otimes 1_{C(\mathbb{T}^d)}, e_2\otimes 1_{C(\mathbb{T}^d)}\right)\in M_\infty(\widetilde{A_n})$ be a projection defined by Equation~\ref{proj_e}. For $m\geq n$, put 
\[\tilde{\Gamma}_{m,n}(e)=\left(\tilde{e}_1\otimes 1_{C(\mathbb{T}^d)}, \tilde{e}_2
  \otimes 1_{C(\mathbb{T}^d)}\right).\]

Then,
 \[\text{rank}(\tilde{e}_2)=\frac{Mr(m)}{r(n)},\] which by the choice of $M$, 
 is strictly less than $r(m)ch_n'+h'_ns'(n)r(m)$. Because 
\[c=\inf_{m\geq 0}\frac{s'(m)}{r(m)} ,\]
we have
  \[h'_n c r(m)\leq h'_n\frac{s'(m)}{r(m)}r(m)=h'_ns'(m).\] 
  
  Hence, for $m\geq n$, we have 
 \begin{equation}\label{equ_rank_e_2}
 \text{rank}(\tilde{e}_2)<h'_ns'(m)+h'_ns'(n)r(m).
 \end{equation}
  Now, for any trace $\tau$ on $\widetilde{A}_m$, we have
  \[\tau(\left(\tilde{p}_{m,n}, \tilde{q}_{m,n}\right))=\frac{h'_ns'(n)r(m)}{r(m)}=h'_ns'(n).\] Moreover, by Inequality \ref{equ_inequality_sec6_}, we get
\begin{equation*}
	\begin{split}
		\tau\left(\tilde{\Gamma}_{m,n}(e)\right)=\frac{M}{r(n)}>h'_ns'(n)+\rho=\tau(\left(\tilde{•}{p}_{m,n}, \tilde{q}_{m,n}\right))+\rho.
	\end{split}
\end{equation*}
If \[\tilde{\Gamma}_{\infty,n}(\left(\tilde{p}_{m,n}, 
\tilde{q}_{m,n}\right))\lesssim \tilde{\Gamma}_{\infty,n}(e),\]   
then there exists some $m>n$ and $x\in M_{\infty}(\widetilde{A_m})$ such that 
\[
\Vert x\tilde{\Gamma}_{m,n}(e)x^*-\left(\tilde{p}_{m,n},\tilde{q}_{m,n}\right) \Vert<\frac{1}{2}.
\]

By Lemma~\ref{lem_sec4_rankbound}, the above inequality implies that 

\[\text{rank}(\tilde{e}_2)\geq h'_ns'(n)r(m)+h'_ns'(m),\] which is a contradiction by Equation \ref{equ_rank_e_2}.

\end{proof}

\bibliographystyle{amsplain}
\bibliography{rc_bibtex}  

\end{document}